\documentclass[12pt]{article}
\usepackage{geometry} 
\usepackage{amsmath}
\usepackage{amssymb}
\usepackage{mathabx}
\usepackage{amscd}
\usepackage{latexsym}
\usepackage{amsthm}
\usepackage{mathrsfs}
\usepackage{color}
\usepackage{amsfonts}
\usepackage{enumitem}
\usepackage{array}
\usepackage{setspace}
\usepackage{mathtools}
\usepackage{tikz}
\usepackage{tikz-cd}
\usepackage{hyperref}
\usetikzlibrary{arrows}
\usetikzlibrary{knots}
\usepackage{ytableau}

\usepackage{stmaryrd,graphicx}

\newtheorem{thm}{Theorem}[section]
\newtheorem*{thm*}{Theorem}
\newtheorem{cor}[thm]{Corollary}
\newtheorem*{cor*}{Corollary}
\newtheorem{lem}[thm]{Lemma}
\newtheorem*{lem*}{Lemma}
\newtheorem{prop}[thm]{Proposition}
\newtheorem*{prop*}{Proposition}

\newtheorem*{thma}{Theorem A}

\theoremstyle{definition}
\newtheorem{defn}[thm]{Definition}
\newtheorem*{defn*}{Definition}

\newtheorem{conjecture}[thm]{Conjecture}
\newtheorem*{conjecture*}{Conjecture}

\newtheorem*{condition*}{Condition}
\newtheorem*{assumption*}{Assumption}

\theoremstyle{remark}
\newtheorem{rem}[thm]{Remark}
\newtheorem*{rem*}{Remark}
\newtheorem{example}[thm]{Example}

\newtheorem*{problem*}{Problem}

\DeclareMathOperator{\move}{move}
\DeclareMathOperator{\diag}{diag}
\DeclareMathOperator{\pfi}{PF}
\DeclareMathOperator{\npfi}{NPF}

\DeclareMathOperator{\area}{area}
\DeclareMathOperator{\red}{red}
\DeclareMathOperator{\dimv}{dinv}

\newcommand{\ratdyck}{\tilde{\pi}}
\newcommand{\areaseq}{a}
\newcommand{\uio}{D}
\newcommand{\coareaseq}{b}

\newcommand{\dleq}{\trianglelefteq}

\newcommand{\invo}{\iota}
\newcommand{\vanset}{\mathcal{A}}
\newcommand{\ddel}{\delta}
\newcommand{\vdinv}{d}
\newcommand{\nablat}[1]{\nabla'_{#1}}

\newcommand{\csf}{\mathcal{X}}
\newcommand{\csfa}{\xi}
\newcommand{\zgz}{\mathbb{Z}_{\geq 0}}

\newcommand{\bruleq}{\leq_{bru}}
\newcommand{\afflag}{\tilde{\mathcal{F}}}

\newcommand{\bF}{\mathbf F}

\newcommand{\BP}{\mathbb P}

\newcommand{\Bk}{\mathbf k}

\newcommand{\CE}{\mathcal E}

\DeclareMathOperator{\sort}{sort}

\newcommand{\trans}[1]{t(#1)}
\newcommand{\wpos}{W^+}

\newcommand{\ol}{\overline}

\newcommand{\aut}{\mathrm{aut}}

\newcommand{\frob}{\mathcal{F}}

\newcommand{\acts}{\righttoleftarrow}

\newcommand{\cO}{\mathcal{O}}
\newcommand{\ii}{a}

\newcommand{\labtocmp}{\alpha}
\DeclareMathOperator{\coarea}{coarea}

\DeclareMathOperator{\affsp}{X}
\DeclareMathOperator{\revl}{rev}
\DeclareMathOperator{\edges}{edges}

\DeclareMathOperator{\paff}{aff}
\DeclareMathOperator{\shuff}{Std}
\DeclareMathOperator{\Hilb}{Hilb}

\DeclareMathOperator{\Hom}{Hom}

\DeclareMathOperator{\Aut}{Aut}

\DeclareMathOperator{\weight}{weight}

\DeclareMathOperator{\Nilp}{Nilp}

\DeclareMathOperator{\pExp}{Exp}

\DeclareMathOperator{\trace}{Tr}

\DeclareMathOperator{\Ext}{Ext}
\DeclareMathOperator{\inv}{inv}
\DeclareMathOperator{\dinv}{dinv}

\newcommand{\pakstan}{{\mathbf w}}
\newcommand{\ratsp}[2]{\affsp_{#1,#2}}
\newcommand{\raths}[2]{\mathcal{H}_{#1,#2}(q,t)}
\newcommand{\ratfrob}[2]{\mathcal{F}_{#1,#2}[Y,X;q,t]}
\newcommand{\pardiag}{T}

\newcommand{\dyckpath}{\pi}

\newcommand{\Ht}{\tilde{H}}

\newcommand{\nlist}{\mathbf}
\newcommand{\an}{{\nlist{a}}}
\newcommand{\bn}{{\nlist{b}}}
\newcommand{\cn}{{\nlist{c}}}
\newcommand{\dn}{{\nlist{d}}}

\newcommand{\mn}{{\nlist{m}}}

\newcommand{\xn}{{\nlist{x}}}
\newcommand{\yn}{{\nlist{y}}}
\newcommand{\zn}{{\nlist{z}}}
\newcommand{\wn}{{\nlist{w}}}
\newcommand{\Z}{\mathbb{Z}}
\newcommand{\C}{\mathbb{C}}

\newcommand{\ppmap}{\phi}

\DeclareMathOperator{\image}{Im}

\DeclareMathOperator{\rank}{rank}
\DeclareMathOperator{\type}{type}
\DeclareMathOperator{\labs}{labs}

\author{Erik Carlsson, Anton Mellit}

\begin{document}
\title{A Combinatorial formula for the nabla operator}
\maketitle

\begin{abstract}

We present an LLT-type formula for a general power
of the nabla operator of \cite{bergeron1999science}
applied to the Cauchy product for the modified Macdonald
polynomials, and use it to deduce a new proof of the
generalized
shuffle theorem describing $\nabla^k e_n$ \cite{haglund2005combinatoriala,carlsson2018proof,mellit2016toric}, 
and the formula for $(\nabla^k p_1^n,e_n)$ from
\cite{elias2016computation,gorsky2017hilbert} as corollaries.
We give a direct proof of the theorem by 
verifying that the LLT expansion satisfies the
defining properties of $\nabla^k$,
such as triangularity in the dominance order,
as well as a geometric proof
based on a method for counting bundles on $\mathbb{P}^1$ due to the second author
\cite{mellit2017poincarea}.
These formulas are related to an affine paving of the 
type A unramified
affine Springer fiber studied by Goresky, Kottwitz,
and MacPherson in \cite{goresky2004unramified},
and also to Stanley's chromatic symmetric functions.
\end{abstract}

\section{Introduction}

There is a well-studied connection between the
combinatorics of the nabla operator of 
\cite{bergeron1999science,bergeron1999identities},
and the homology or cohomology of the
affine Springer fibers $\affsp_{\gamma}$ of the sort
studied in \cite{goresky2003purity},
see for instance
\cite{lusztig1991fixed,goresky2004unramified,gorsky2013compactified,gorsky2014torus,gorsky2014affine,hikita2014affine,oblomkov2014geometric,carlsson2018affine,kivinen2020unramified}. 
In this picture, objects such as parking functions $(\pi,w)$
are seen to
be in bijection with cells in an affine paving of $\affsp_\gamma$,
and combinatorial statistics such as $\dinv(\pi,w)$ 
that appear in the shuffle
theorem \cite{haglund2008catalan,haglund2005combinatoriala,carlsson2018proof} and other nabla-type formulas are essentially the dimensions of 
the corresponding cells. In this way, combinatorial formulas
may be interpreted as graded characters of the homology
of some $\affsp_\gamma$, with the $q$-degree representing
half the homological degree, the $t$-degree being more subtle.

For example, consider the following
power series
\begin{equation}
\raths{m}{n}=\label{eq:raths}
\frac{1}{(1-q)^{\gcd(n,m)}}\sum_{\substack{w\in \wpos_n
\\ \text{$w^{-1}$ is $m$-stable}}}
t^{\area(w)} q^{\dimv_m(w)}.
\end{equation}
Here $\wpos_n$ is a set of extended affine permutations
in which $w_i\geq 1$ for $1\leq i \leq n$, $v$ is
$m$-stable if $v_{i+m}> v_i$ for all $i$, and $\area$
and $\dinv_m$ are defined in 
Section \ref{sec:aff}.
A slightly different version of this series was
presented in \cite{gorsky2014affine} 
in the case when $n,m$ are relatively prime, which 
the authors showed describes the combinatorics of the rational
version of the shuffle theorem \cite{bergeron2016compositional,mellit2016toric}.
The corresponding Springer fiber in this case is
$\affsp_\gamma=\ratsp{n}{m}\subset \afflag_n$, which 
is the one associated to the nil-elliptic
operator $\gamma=N^m$, where $N(e_i)=e_{i+1}$ for $i\leq n$,
$N(e_n)=te_1$.

There is an extension of $\ratsp{n}{m}$ for $n,m$
not relatively prime, which in the case of 
$m=kn$, becomes the unramified affine Springer fiber
studied in \cite{goresky2004unramified}.
In this case, the equivariant homology 
$H_*^T(\affsp_{n,kn})$ for a standard
torus action $T\acts \affsp_{n,kn}$
is equipped with two commuting actions of the symmetric
group, generally known as the ``dot'' and ``star'' actions which
act on the left and right respectively,
due to Knutson and Tymoczsko
\cite{knutson2003schubert,tymoczko2008permutation}.
The dot action comes from a  space-level action
on the affine flag variety, which permutes different
fibers $\affsp_{n,kn}$, whereas the right action comes
from the Springer action.
In this paper, we present an LLT-type expansion
\[\Omega_k[X,Y]=
\sum_{[\mn, \an, \bn]}  \frac{t^{|\mn|} 
q^{\dinv_k(\mn, \an,\bn)}}{(1-q)^n \aut_q(\mn,\an,\bn)}
X_{\an} Y_{\bn},\]
for $k\geq 1$, 
where the quantities in the summand are defined in Section \ref{sec:comb}.
We predict that $\Omega_k[X,Y]$
corresponds to a Frobenius 
character extension of \eqref{eq:raths}, namely
$\Omega_k[X,Y]=\ratfrob{n}{kn}$,
where the $Y$-variables represent the dot action,
and the $X$-variables correspond to star.
In particular, the coefficient of the monomial
with all exponents equal to one in $\Omega_k[X,Y]$
is shown to agree with $\raths{n}{kn}$.

Our main theorem is that $\Omega_k[X,Y]$ is
computed by powers of the 
$\nabla$-operator applied to the Cauchy product
for the modified Macdonald polynomials, shown in plethystic
notation:
\begin{thma}
\label{thm:thma}
For $k\geq 1$, we have
\[\nabla^k e_n\left[\frac{XY}{(1-q)(1-t)}\right] 
= \Omega_k[X,Y].\]
\end{thma}
Notice that unlike most combinatorial
formulas involving the nabla operator
formulas, the one in Theorem A completely 
determines $\nabla^k$, and could
therefore be taken as a definition.
This is the key point to our first
proof, which is done by verifying
that $\Omega_k[X,Y]$ satisfies the
defining properties of $\nabla^k$,
similar to the approach taken in 
\cite{haglund2005combinatorial}.
We give a second proof based on a
method developed by the second
author for counting bundles on 
$\mathbb{P}^1$ over a finite field,
which we hope will lead to further
connections with geometry and number theory.

We also deduce as corollaries some well-known
formulas involving the $\nabla$-operator, namely
the generalization of the shuffle theorem 
\cite{haglund2005combinatoriala,carlsson2018affine}
for arbitrary powers $\nabla^k e_n$, 
and the Elias-Hogancamp expression
for $\nabla^k p_1^n$ \cite{elias2016computation,gorsky2017hilbert}.
Our primary motivation for studying this formulas is
part of an ongoing study of Tor groups of
certainly polygraph-type modules, in connection with
the nabla positivity conjecture 
of Bergeron, Garsia, Haiman, and Tesler \cite{bergeron1999identities}, which predicts that
the coefficients of the Schur expansion of 
$\nabla^k s_\lambda$ are polynomials in $c_{\lambda,\mu}(q,t)$
whose coefficients are entirely positive or entirely negative,
which we will address in future papers.

\emph{Acknowledgments}. E. Carlsson was supported by
NSF DMS-1802371 during part of this project. A. Mellit was supported by the projects Y963-N35 and
P31705 of the Austrian Science Fund.

\section{Premilinary definitions and notations}
In this section we give general background on plethysm,
affine permuations, and the combinatorial 
constructions that appear in our main theorem.

\subsection{Macdonald polynomials}
\label{sec:mac}

Given a symmetric function $f$, we will adopt the usual plethysm
notation of $f[X]$ when $X$ is an element of some $\lambda$-ring,
so that $f[x_1+\cdots+x_N]$ is the substitution
$f(x_1,...,x_N)$. If $X=(x_1,x_2,...)$ is some alphabet, we
will use the same letter $X$ to denote the sum in plethystic formulas.
For details, we refer the
reader to \cite{Haiman01vanishingtheorems}.

Let $\Ht_\lambda=\Ht_\lambda(X;q,t)$ 
denote the modified Macdonald polynomial
\cite{bergeron1999identities}, defined by
\[\Ht_\lambda(X;q,t)=
t^{n(\lambda)} J_{\lambda}[X/(1-t^{-1});q,t^{-1}].\]
Let $\nabla$ be the Garsia-Haiman-Bergeron-Tesler operator
\begin{equation}
\label{eq:nabladef}    
\nabla \Ht_\lambda(X;q,t)=q^{n(\lambda')} t^{n(\lambda)} \Ht_\lambda(X;q,t),
\end{equation}
where
\[n(\lambda)=\sum_{i} (i-1)\lambda_i\]
is the usual statistic from Macdonald's book
\cite{macdonald1995symmetric}. 
In this paper, $\nabla$ will always denote an operator
applied to the $X$ variables.

\subsection{Combinatorial definitions}
\label{sec:comb}
Fix $n$ and define a label to be an
$n$-tuple of positive integers
$\an =(a_1,...,a_n)$ with $a_i\geq 1$. 
We will write $\labs(n)$ for the set of all labels of length
$n$, and will also call the individual $a_i$ labels.
For any label $\an$, we have a multiset 
$A=A(\an)=(|A|,m_A)$ where
$|A|=\{a_1,...,a_n\}$ is the total set,
and $m_A:|A|\rightarrow \Z_{\geq 1}$ is the multiplicity.
We define a (strict) composition of $n$
\[\labtocmp(\an)=(\alpha_1,...,\alpha_{l}),\quad
|A|=\{c_1<\cdots < c_l\},\quad \alpha_i=m_A(c_i).\]
In other words, $\alpha(\an)$ is the result of sorting $\an$ in increasing order, and reading of the sizes of the groups, for instance
\[\alpha((1,1,1,4,4,2,1,4))=(4,1,3).\]
We may also define the corresponding partition 
$\mu(\an)=\mu(\alpha(\an))$ which
is the result of sorting $\alpha(\an)$ in decreasing order, so
$\mu(\an)=(4,3,1)$ in the above example.
Given a multiset $A$, let $\labs(A)$ denote the 
set of labels $\an$ with $A(\an)=A$, with similar
definitions for $\labs(\alpha)$ and $\labs(\mu)$.

If $\mathcal{A},\mathcal{B},...$ are totally ordered sets, 
we define the ordering on $\mathcal{A}\times \mathcal{B}\times \cdots $
as the corresponding lexicographic order, breaking ties from left
to right. If $\an \in \mathcal{A}^n$, $\bn \in \mathcal{B}^n$,... 
are some elements, we define $[\an,\bn,...]$ to be the
sorted representative of the simultaneous action of $S_n$ on
all components.
In other words, view $(\an,\bn,...,)$ as a matrix,
transpose the matrix, sort according to the order on 
$\mathcal{A}\times \mathcal{B}\times \cdots $, and transpose back.
For instance, in the case $\an\in \mathcal{A}^n$,
$\bn\in \mathcal{B}^n$ for $\mathcal{A}=\mathcal{B}=\mathbb{Z}_{\geq 1}$,
we have
\[[(1,2,1,1,2,1),(3,2,3,1,1,3)]=((1,1,1,1,2,2),(1,3,3,3,1,2)).\]
We can then define $\alpha(\an,\bn,...)$ using the same rules as
above, so in the above example $\alpha(\an,\bn)=(1,3,1,1)$.
We make a similar definition for $\mu$, which also applies when 
the sets are unordered.

\subsection{The dinv statistic}
\label{sec:dinv}

Let $\an,\bn$ be labels,
let $\mn \in \mathbb{Z}_{\geq 0}^n$, 
with the decreasing order on the $m_i$,
so that a triple $[\mn,\an,\bn]$ means one sorted as in the following way.
\begin{defn} 
\label{def:mab}
Let $\mn \in \zgz^n$ and let $\an,\bn$ be labels.
We will say that $(\mn,\an,\bn)$ is sorted if
for every $i<j$ we have
\begin{enumerate}
    \item $m_i \geq m_j$, and
\item if $m_i=m_j$ then $a_i\leq a_j$, and
\item if $m_i=m_j$ and $a_i=a_j$ then $b_i\leq b_j$.
\end{enumerate}
\end{defn}
For instance,
\[[(1,0,1,0),(2,1,1,1),(1,2,2,1)]=((1,1,0,0),(1,2,1,1),(2,1,1,2)).\]
We will often write such lists as arrays, as in
Example \ref{ex:dinvdp} below.

We now define a statistic $\dinv_k(\mn,\an,\bn)$ on triples
which are sorted according to Definition \ref{def:mab}. 
\begin{defn}
\label{defn:dinv}
Let $\mn\in \mathbb{Z}_{\geq 0}^n$,
let $\an,\bn\in \labs(n)$ be labels,
and suppose that $(\mn,\an)$ are sorted. We define
\begin{subequations}
\label{eq:dinvdef}
\begin{equation}
    \label{eq:dinvdefa}
    \dinv_k(\mn,\an,\bn)=\sum_{i<j} \dinv^{i,j}_k(\mn,\an,\bn)
\end{equation}
where
\begin{equation}
\label{eq:dinvdefb}    
\dinv_k^{i,j}(\mn,\an,\bn)=\max\left(
m_j-m_i-1+k+\delta(a_i>a_j)+\delta(b_i>b_j),0\right),
\end{equation}
\end{subequations}
and $\delta(a_1>a_2)$ is one if $a_1>a_2$, zero otherwise.
\end{defn}
We similarly define $\dinv_k(\mn,\an)$ as the result
of removing
$\delta(b_i>b_j)$, which is the
same as setting $\bn=(1^n)$
by default.

Recall that a Dyck path is a path 
of North and East steps in the 
$n\times n$ grid beginning at the origin
$(0,0)$, placed in the South-West,
or lower left corner, and ending
at $(n,n)$, which never goes
below the diagonal. It is determined
uniquely by the set
\[\uio(\pi)=\left\{
(i,j) : \text{$1\leq i<j\leq n$
is between the path and the diagonal}
\right\}.\]
\begin{defn}
\label{defn:dinv1}
Fix $k\geq 0$, 
suppose $(\mn,\an)$ is sorted, and let $i<j$.
We will say that $i$ $k$-attacks $j$ 
(or just attacks) if
\[m_j-m_i-1+k+\delta(a_i>a_j)\geq 0.\]
\end{defn}
In other words, $i$ $k$-attacks $j$ if
switching the order of $b_i,b_j$ has an effect 
on $\dinv_k$.
For instance, for $k=1$ we have
that $i$ attacks $j$ if
\begin{enumerate}
    \item $m_i=m_j+1$ and $a_i>a_j$, or
\item $m_i=m_j$.
\end{enumerate}
Note that in the second condition,
we necessarily have $a_i\leq a_j$.
\begin{defn}
\label{def:pima}
    Let $\dyckpath=\dyckpath_k(\mn,\an)$ denote the Dyck path such that the elements of 
    $\uio(\dyckpath)$, 
are the pairs $i<j$ for which 
$i$ $k$-attacks $j$.
\end{defn}
We now have that
\begin{equation}
\label{eq:dinv}
\dinv_k(\mn, \an, \bn) = 
\dinv_k(\mn,\an)+\inv_{\dyckpath_k(\mn,\an)}(\bn)
\end{equation}
where
\begin{equation}
    \label{eq:invpi}
  \inv_\dyckpath(\bn)  =
    \#\{(i,j)\in \uio(\dyckpath) : b_i>b_j\}.
\end{equation}

\begin{example}
\label{ex:dinvdp}
Let $\mn,\an$ be given in array notation by
\[\left(\begin{array}{c|cccccc}
\mn&3&3&3&2&0&0\\
\an&1&1&5&4&2&5\\ 
\end{array}\right),\]
which is a sorted term for $n=6$. Then we find that
$\pi_2(\mn,\an)$ is the Dyck path given in Figure \ref{fig:dyckpath66},
as the attacking pairs are the elements of $\uio(\pi)$ listed in the caption.
\end{example}
\begin{figure}
\begin{centering}
\begin{tikzpicture}
\draw[help lines] (0,0) grid (6,6);
\draw[dashed,color=gray] (0,0)--(6,6);
\draw[-,very thick] (0,0)--(0,1);
\draw[-,very thick] (0,1)--(0,2);
\draw[-,very thick] (0,2)--(0,3);
\draw[-,very thick] (0,3)--(0,4);
\draw[-,very thick] (0,4)--(1,4);
\draw[-,very thick] (1,4)--(2,4);
\draw[-,very thick] (2,4)--(3,4);
\draw[-,very thick] (3,4)--(3,5);
\draw[-,very thick] (3,5)--(4,5);
\draw[-,very thick] (4,5)--(4,6);
\draw[-,very thick] (4,6)--(5,6);
\draw[-,very thick] (5,6)--(6,6);
\end{tikzpicture}
\caption{A Dyck path of size $(6,6)$
with area sequence
$\areaseq(\pi)=(0, 1, 2, 3, 1, 1)$,
and $\uio(\pi)=\{(1,2),(1,3),(1,4),(2,3),(2,4),(3,4),(4,5),(5,6)\}$.}
\label{fig:dyckpath66}
\end{centering}
\end{figure}
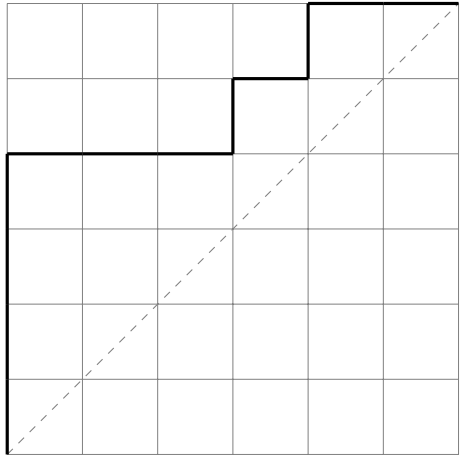

\subsection{Examples}
\label{sec:combex}
A sum over all $\an$ will mean the infinite sum
over all labels, unless some upper bound is specified, $a_i\leq N$.
We will adopt a convenient convention that a sum over
$[\an,\bn,...]$ means a sum over cosets, with the assumption that
$(\an,\bn,...)$ is the sorted representative in the summand.
We will also allow for some summands in which only some of the
summands are grouped, which means that just those terms are sorted. For instance, the symbol
\[\sum_{[\an,\bn],\cn,[\dn]}\cdots\]
indicates the sum over quadruples $(\an,\bn,\cn,\dn)$ so that
for every $i<j$ we have
$a_i\leq a_j$, $b_i\leq b_j$ if $a_i=a_j$, $d_i\leq d_j$,
and there are no constraints on $\cn$.
We also define automorphism factors for the cosets
\[\aut(\an,\bn,...)=\prod_i \mu_i!,\quad \aut_q(\an,\bn,...)
=\prod_i [\mu_i]_q!,\]
where $\mu=\mu(\an,\bn,...)$, and 
\[[k]_q=1+q+\cdots q^{k-1},\quad
[k]_q!=\prod_{j=1}^k [j]_q,\]
are the $q$-number and $q$-factorial.

We give some examples in symmetric functions. Let
\[X_\an=x_{a_1}\cdots x_{a_n}=\prod_{a\in A} x_a^{m(a)}\]
be the associated monomial to $\an$, where $(A,m)$ is the 
associated multiset.

\begin{example}
\label{ex:complete1}
The complete and monomial symmetric functions are given by
\[h_n(x_1,x_2,...)=
\sum_{\an} \frac{1}{\aut(\an)} X_{\an}=
\sum_{[\an]} X_{\an}=\sum_{\mu} m_\mu(x_1,x_2,...),\] 
\[m_{\mu}(x_1,x_2,...)
=\frac{1}{\mu_1!\cdots \mu_l!}
\sum_{\mu(\an)=\mu} X_{\an}.\]
We also have the quasi-symmetric monomials defined by
\[m_\mu=\sum_{\mu(\alpha)=\mu} M_\alpha,\quad
M_\alpha(x_1,x_2,...)
=\frac{1}{\alpha_1!\cdots \alpha_l!} \sum_{\alpha(\an)=\alpha}
X_{\an}.\]
\end{example}
\begin{example}
We have
\[e_n\left[\frac{X}{1-q}\right]=
\sum_{[\an]} \frac{q^{n(\mu(\an)')}}{(1-q)^n \aut_q(\an)} X_{\an},\]
which follows from using the Cauchy product, and
the well-known specializations for 
$h_\mu(1,q,...)$. Replacing $e_n$ with $h_n$
simply removes the $q^{n(\mu(\an)')}$ factor.
\end{example}

\begin{example}
\label{ex:cauchymab}
We have the Cauchy product for the modified Macdonald polynomials
from Section \ref{sec:mac}:
\begin{equation}
\label{eq:hnqt}    
e_n\left[\frac{XY}{(1-q)(1-t)}\right]=
\sum_{[\mn,\an,\bn]} \frac{t^{|\mn|}q^{n(\mu(\an)')}}{(1-q)^n \aut_q(\mn,\an,\bn)}
X_{\an}Y_{\bn}.
\end{equation}
The order we chose for 
$\mn$ does not affect the answer here,
but it does make a difference in
Theorem \ref{thm:thm1},
which is about powers of the nabla operator applied to this
expression. Again, replacing $e_n$ with $h_n$
simply removes the $q^{n(\mu(\an)')}$.
\end{example}

\section{Motivation}
We explain the underlying motivation behind Theorem A,
which was discovered experimentally
using conjectural relations between Haiman's polygraph rings
and the homology of the unramified affine Springer fiber
$H^T_*(\affsp_{n,kn})$.
We connect the combinatorics of Section \ref{sec:comb} to 
cells in $\affsp_{n,kn}$.

\subsection{Affine permutations}
\label{sec:aff}
We describe the connection between affine permutations
and the combinatorics of the $\dinv$ statistic and rational
slope parking functions, following \cite{gorsky2014affine}.

Define the set of positive affine permutations as
\[\wpos_n=\left\{w: \mathbb{Z}\rightarrow \mathbb{Z}:
w(i+n)=w(i)+n, \mbox{$w_i\geq 1$ for $i\geq 1$}\right\}.\] 
Each one is determined by its values in window notation,
$w=(w_1,...,w_n)$.
Notice that these all have positive numbers as entries,
and are not normalized in the usual way so that the sum 
of the entries is $n(n+1)/2$, which is a condition for 
affine $SL_n$ as opposed to $GL_n$. 
Let
\[\wpos_{n,d}=\left\{w\in \wpos_{n} :w_1+\cdots+w_n=dn+n(n+1)/2\right\}.\]
We may still multiply any two such permutations, which results
in adding the values of $d$. For each $d$ we have the Bruhat
order $\leq_{bru}$ on $\wpos_{n,d}$. 

Following \cite{gorsky2014affine}, we have
\begin{defn}
\label{def:stable}
An affine permutation is called $m$-stable if
$w_{i+m}>w_i$ for all $i$, and is called $m$-restricted
if $w^{-1}$ is $m$-stable. 
\end{defn}
These are the fixed points of the affine Springer fiber
$\affsp_{n,m}=\affsp_{\gamma}$
of the type studied in \cite{goresky2003purity},
in which $\gamma$ is the topologically nilpotent operator
\[\gamma (e_i)=a_{i} e_{i+m}.\]
Here $e_i \in \C^n((t))$ is the standard basis vector in $\C^n$
for $1\leq i \leq n$, and otherwise $e_{i+n}=te_i$,
and the $a_i$ are distinct nonzero complex numbers for for 
$1\leq i \leq d$, and $a_{i+d}=a_i$ for $d=\gcd(n,m)$.
For $m=kn$, we have
$\gamma=\diag(a_1t^k,...,a_nt^k)$,
corresponding to the unramified
case studied in \cite{goresky2004unramified}.
In this case the 
$n$-dimensional torus $T\subset GL_n(\C)$ acts by multiplication
on the left via, as well as the extended $(n+1)$-dimensional torus
$\widetilde{T}$, which includes loop rotation, both having discrete
fixed points described by Definition \ref{def:stable}.

For integers $a,b \in \Z$ which are not congruent modulo $n$,
we have an affine transposition $t_{a,b}$ which switches the two. 
Given an $m$-restricted permutation, let
\begin{equation}
    \edges_m(w)=
    \left\{t_{a,b} : t_{a,b} w\leq_{bru}w,\ 
    |a-b|<m\right\}.
\end{equation} 
The statistic $|a-b|$ does not
depend on the representatives $a,b$ or their order,
and is called the height of the 
transposition.
The set $\edges_m(w)$ represent directed edges 
$w\rightarrow v$ with $v=t_{a,b}w$ in the GKM graph
of $\affsp_{n,m}$, corresponding to the one-dimensional
orbits under $\widetilde{T}$.

Recall that an $(n,m)$-rational slope Dyck path is one that begins
at $(0,0)$ and ends at $(m,n)$, never crossing the
line of slope $n/m$. Again, we have the area and coarea sequences
$\area(\pi),\coarea(\pi)$,
and also $\uio(\pi)$.
For any $m$-restricted permutation $w$, 
there is a rational $(n,m)$-Dyck path
with coarea sequence
\[\coareaseq(\ratdyck_{m}(w))=\sort(\pakstan_m(w),<),\]
where 
\[\pakstan_m(w)_j =\#\left\{t_{a,b} \in \edges_m(w):
w^{-1}t_{a,b}w=t_{i,j} \text{ for some $i<j$}\right\}.
\]
This is the underlying Dyck path of sequence 
$\pakstan_m(w)=\mathcal{PS}_{w^{-1}}$ of \cite{gorsky2014affine},
which is shown to define a bijection from the set of 
$m$-stable affine permutations in $W_n$
to rational parking functions for $(n,m)$ coprime.

We define
\begin{defn}
    \label{def:dimv}
    Let
    \[\dimv_m(w)=\area(\pi_{n,m})-\#\edges_m(w)=\area(\ratdyck_m(w)),\]
    where $\pi_{n,m}=(1^n0^m)$ is the
    $(n,m)$-Dyck path of maximal area.
\end{defn}

We now explain the connection with the
statistic $\dinv_k(\mn,\an,\bn)$ of Section \ref{sec:comb}.
Recall the definition of 
Standardization from \cite{haglund2005combinatorial}:
\begin{defn}
\label{def:standardization}
The standardization of a label
is the unique permutation $\sigma=\shuff(\an)$ such that
$\an_{\sigma^{-1}}$ is weakly increasing, and
the restriction of $\sigma$ to $\an^{-1}(\{x\})$
is increasing if $x$ is positive, decreasing
if $x$ is negative. 
\end{defn}
We will also define $\shuff_<(\an)$
and $\shuff_>(\an)$ with respect to the usual, and reverse
order on $\mathbb{Z}_{\geq 1}$, so that $\shuff=\shuff_<$.
For instance, if $\an=(3,3,3,1,2,3,1)$, then 
\[\shuff_<(\an)=
(4, 5, 6, 1, 3, 7, 2),\quad 
\shuff_>(\an)=
(1, 2, 3, 6, 5, 4, 7).\]
In particular, $\dinv_k$ respects
standardization, i.e.
$\dinv_k(\mn,\shuff(\an),\shuff(\bn))=
\dinv_k(\mn,\an,\bn)$.

Now given a tuple $(\mn,\an,\bn)$ which is sorted, we
define an affine permutation
\begin{equation}
    \label{eq:paff}
    \paff(\mn,\an,\bn)=
    \shuff_{>}(\revl(\bn)) \trans{\mn}
    \shuff_{<}(\an)^{-1},
\end{equation}
where
$\trans{\mn}=(n+m_1 n,...,1+m_nn)$ is the maximal representative
of its coset in $S_n\backslash \wpos_n /S_n$,
and $\revl(\bn)$ is the result of writing
$\bn$ in the reverse order.
We similarly define $\paff(\mn,\an)$ as the left coset
$S_n \paff(\mn,\an,\bn)$, which is independent of $\bn$.
The proof of the following proposition is tedious,
and will be omitted.
\begin{prop}
\label{prop:paff}
Fix multisets $A,B$ of size $n$
with $|A|,|B|\subset \mathbb{Z}_{\geq 1}$.
Then 
$\paff(\mn,\an,\bn)$ defines a bijection from the set of sorted triples 
\[(\mn,\an,\bn)\in \zgz^n \times \labs(A)\times \labs(B)\]
to the double coset 
$S_{\revl(\alpha(B))}\wpos_n S_{\alpha(A)}$,
where $S_\alpha$ is the Young subgroup,
and $\paff(\mn,\an,\bn)$ is the unique representative of its
double coset of maximal length.
Moreover, we have
\[\dinv_k(\mn,\an,\bn)=
\dimv_{kn}\left(\paff(\mn,\an,\bn)\right),\]
and the Dyck path is determined by
\[\areaseq(\dyckpath_k(\mn,\an))=\areaseq(\ratdyck_{kn}(w_{min}))-
\areaseq(\ratdyck_{kn}(w_{max})),\]
where $w_{min},w_{max}$ are the unique representatives of the
coset $\paff(\mn,\an)$ which are minimal and maximal in the Bruhat order.
\end{prop}

\begin{example}
Take $(\mn,\an,\bn)=((2,1,0,0),(2,3,1,1),(1,2,1,1))$, which is sorted.
Then we have
\[\trans{\mn}=(12,7,2,1),\quad \shuff_<(\an)=(3, 4, 1, 2),\quad \shuff_>(\revl(\bn))=(2,3,1,4),\]
which gives $w=\paff(\mn,\an,\bn)=(3,2,12,5)$. 
This is the maximal length element in the double coset
\[S_{(1,3)} w S_{(2,1,1)}=\left\{
(2, 3, 12, 5), (2, 4, 11, 5),\right.\]
\[\left.(3, 2, 12, 5), (3, 4, 10, 5), (4, 2, 11, 5),
    (4, 3, 10, 5)\right\}.\]
Now for $k=1$, we have
\[\edges_{4}(w)=
\left\{
(2, 3, 12, 5), (3, 2, 9, 8), (4, 2, 11, 5), (3, 4, 10, 5)
\right\}\]
so that
$\dimv_{4}(w)=6-4=2$.
On the other hand, $(\mn,\an)$ has
three attacking pairs, $\left\{(2,3),(2,4),(3,4)\right\}$.
Since $(2,3)$ and $(2,4)$ are the pairs for which $b_i>b_j$,
we see that $\dinv_1(\mn,\an,\bn)=2$, 
in agreement with Proposition \ref{prop:paff}.
\end{example}

\begin{example}
Let us compute the Dyck path
for the terms $\mn,\an$ from Example
\ref{ex:dinvdp}, using Proposition \ref{prop:paff}.
Then we have that
\[w_{min}=(19, 20, 5, 16, 21, 6), \quad
w_{max}=(24, 23, 2, 15, 22, 1)\]
 are the minimal and maximal representatives of the left coset
 of $\paff(\mn,\an)\in S_n\backslash \wpos_n$.
 Then
 \[\areaseq(\pakstan_{12}(w_{min}))=(0, 2, 4, 4, 1, 2),
 \quad \areaseq(\pakstan_{12}(w_{max}))=(0, 1, 2, 1, 0, 1),\]
and the unique $(n,n)$-Dyck path
 whose area sequence is the difference $(0, 1, 2, 3, 1, 1)$
 is the expected one from Figure \ref{fig:dyckpath66}.
\end{example}

\subsection{Polygraphs and the Hilbert scheme}
\label{sec:sf}
If $M$ is a representation of $S_n\times \cdots \times S_n$ with
$k$ factors,
we will denote the Frobenius character by 
\[\frob_{X_1,...,X_k} M \in \C[x_{i,j}]^{S_n\times \cdots \times S_n},\]
which is a function in $k$ sets of variables,
$X_i=(x_{i,1},x_{i,2},...)$, individually symmetric in 
each one. 
For doubly graded modules, 
the Frobenius character encodes
the degrees with the $q,t$ variables, namely
\[\frob M=\sum_{i,j} q^it^j \frob M^{(i,j)}\]
where $M^{(i,j)}$ is
the homogeneous component of the bigrading.

In Haiman's theory \cite{Haiman01vanishingtheorems},
the sum $\Omega_k[X,Y]$ is the equivariant index
of a sheaf on the Hilbert scheme of points
in the complex plane $\Hilb_n \C^2$, with respect to 
the usual torus action $T=\{(q^{-1},t^{-1})\}\acts \Hilb_n \C^2$.
Let $P$ be the Procesi bundle of rank $n!$
whose fibers carry an action of $S_n$ isomorphic to the
regular representation.
The modified Macdonald polynomial is the Frobenius
character
$\Ht_\lambda=\frob P\big|_{\lambda}$ of the fibers 
of torus-fixed points, which are
isomorphic to the Garsia-Haiman module. Then we have
\begin{equation}
    \label{eq:nabcauchy}
     \nabla^k e_n\left[\frac{XY}{(1-q)(1-t)}\right]=
    \sum_{i}(-1)^i \frob_{Y,X} R^i\Gamma\left(P\otimes P^{*}\otimes \mathcal{L}^k\right).
\end{equation}

Now recall Haiman's polygraph modules \cite{haiman2001hilbert}:
Fix $n$ and let $\xn$ denote the set of variables $(x_1,...,x_n)$,
for some variable $x$. 
Let $\C[\xn,\yn]\cdot S_n$ denote the free left $\C[\xn,\yn]$-module
with one free generator for each permutation $\tau\in S_n$.
Consider the following variant of Haiman's map from \cite{haiman2001hilbert}
equation (152):
\begin{equation}
\label{eq:ppmap}
  \ppmap : 
  \C[\xn,\yn,\zn,\wn]\rightarrow \C[\xn,\yn]\cdot S_n,\quad
g(\xn,\yn,\zn,\wn)\mapsto \sum_{\tau \in S_n} g(\xn,\yn,\tau(\xn),\tau(\yn))\tau.
\end{equation}
We define a module $M$ as the image of
$\ppmap$, 
as a $\C[\xn,\yn]$-module.
We have the usual bidgrading on $M$ compatible with
the grading on the ring $\C[\xn,\yn,\zn,\wn]$,
in which the degree of the $\xn,\zn$ variables are
$(1,0)$, and the $\yn,\wn$ variables have degree $(0,1)$.
Note that $\xn,\yn$ have nothing to do with the symmetric function variables $X,Y$.

There is an action of $S_n\times S_n$ on $M$, which may
also be interpreted as a
commuting left and right action by
\[(\sigma_1,\sigma_2)\cdot f(\xn,\yn)\tau=
\sigma_1 \cdot f(\xn,\yn)\tau \cdot \sigma_2^{-1}=
  f(\sigma_1(\xn),\sigma_1(\yn)) \left(\sigma_1 \tau \sigma_2^{-1}\right).\]
Then $\ppmap$ intertwines this action with the 
one where the first factor simultaneously permutes $\xn,\yn$, and the 
second factor permutes $\zn,\wn$.
Notice that the left $S_n$-action on $M$ is compatible with the
action by permuting the variables, whereas the right $S_n$-action
does not act on the variables.
Another way to say this is that $M$ is a bigraded
module over the smash product $\C[\xn,\yn]\rtimes S_n$,
which is the noncommutative ring by adjoining a generator for each 
$\sigma \in S_n$ with the relation
\[\sigma x_i= x_{\sigma_i} \sigma,\quad 
\sigma y_i= y_{\sigma_i},\]
and that the right action of $S_n$ acts by
automorphisms of $M$.
\begin{conjecture}
\label{conj:bkr}
As a module over $\C[\xn,\yn] \rtimes S_n$, 
$M$ is the image of
the Procesi bundle $P$ under
the Haiman-Bridgeland-King-Reid isomorphism
\[F \mapsto R\Gamma_{\Hilb_n}(P\otimes F).\]
The higher derived functors $R^i\Gamma(P\otimes P)$
vanish, and so $M\cong \Gamma_{\Hilb_n}(P\otimes P)$.
Moreover, we have that $M$ is free 
when regarded as a module
over $\C[\xn]$, in other words forgetting the
$\C[\yn]$-action.
\end{conjecture}
\begin{rem}
Haiman identified the module in the case 
$F= B^{\otimes l}$ as the polygraph module $R(n,l)$
defined in \cite{Haiman01vanishingtheorems}, where
$B$ is the tautological bundle, and
he proved the second two statements for $R(n,l)$.
To the best of our knowledge Conjecture \ref{conj:bkr}
is not known. We point out that the vanishing
statement is definitely false for 
three powers of the Procesi bundle $P^{\otimes 3}$,
which may be seen by observing the Atiyah-Bott localization
actually has negative terms.
\end{rem}
\begin{rem}
The conjecture is motivated by the following geometric picture. Recall the commutative diagram (\cite{haiman2001hilbert})
\[
\begin{tikzcd}
X_n \arrow{r} \arrow{d}{\dyckpath} & \C^{2n} \arrow{d}\\
\Hilb_{n}\C^2 \arrow{r} & \C^{2n}/S_n
\end{tikzcd}
\]
The diagram is a reduced cartesian product and the space $X_n$ is the \emph{isospectral Hilbert scheme}. The map $\dyckpath$ is finite and $P=\pi_* O_{X_n}$. Thus the ring $\Gamma_{\Hilb_n}(P\otimes P)$ is the ring of functions on $X_n\times_{\Hilb_n\C^2} X_n$, which is a closed subscheme of $X_n\times X_n$. On the other hand, $X_n\times_{\Hilb_n\C^2} X_n$ is reduced, so it coincides with the reduced fiber product 
\[
X_n\times_{\Hilb_n\C^2} X_n = \left(\Hilb_n \times_{\C^{2n}/S_n} \left(\C^{2n}\times_{\C^{2n}/S_n} \C^{2n}\right)\right)_{\red}.
\]
The space $\C^{2n}\times_{\C^{2n}/S_n} \C^{2n}$ is covered by graphs of permutations viewed as maps $\C^{2n}\to \C^{2n}$. This induces a covering of $X_n\times_{\Hilb_n\C^2} X_n$ by $n!$ copies of $X_n$. Passing to the rings of functions we obtain ring homomorphisms:
\[
\Gamma(\cO_{X_n}) \otimes \Gamma(\cO_{X_n}) \to \Gamma(\cO_{X_n\times_{\Hilb_n\C^2} X_n}) \to \bigoplus_{\sigma\in S_n} \Gamma(\cO_{X_n}),
\]
whose composition is the map $\phi$ of \eqref{eq:ppmap}. The second map above is injective because the functor $\Gamma$ is left exact. If we knew that the first map is surjective, we would have
\[
\Gamma_{\Hilb_n\C^2}(P\otimes P) = \Gamma(\cO_{X_n\times_{\Hilb_n\C^2} X_n}) = \image \phi.
\]
The conjecture is then reduced to the  vanishing of the higher cohomologies of the ideal sheaf of $X_n\times_{\Hilb_n\C^2} X_n$ in $X_n \times X_n$.
\end{rem}

\subsection{Relation with GKM theory}
\label{sec:sphilb}

The space $\affsp_{n,m}$ has an paving by affines
by the results of $\cite{lusztig1991fixed}$ in the coprime
case, and \cite{goresky2003purity} for the general case,
including $\affsp_{n,m}$ for general $(n,m)$ in type A.
In the unramified case of $m=kn$ studied in \cite{goresky2004unramified},
the equivariant cohomology $H_T^*(X_{n,kn})$ is a submodule 
of the free $\C[\xn]$-module with basis $W_n$, and 
there is a basis in which the leading coefficient in the 
Bruhat order is
\[a_{w,k}(\xn)=\prod_{i<j} (x_i-x_j)^{\#\edges_{kn}^{i,j}(w)}.\]
Here $\edges^{i,j}_{m}(w)$ is the set of transpositions
$t_{a,b}\in \edges_m(w)$ for which 
$\{\bar{a},\bar{b}\}=
\{\bar{i},\bar{j}\}$, where the bar is the congruence class
modulo $n$. The standard description of the corresponding
homology is given as a subspace of the  $\C(\xn)$-vector space
with the same fixed point basis, which is different from
the description of \cite{goresky2004unramified}.
For instance, in the case of $k=\infty$, 
we can compare
the coefficient $a_{w,\infty}(\xn)$ with the leading terms
in Kostant and Kumar's nil Hecke ring \cite{kostant1986nil,lam2014schur},
which encodes the equivariant homology of the affine flag variety
\[A_{w}=\sum_{v\bruleq w} c_{v,w}(\xn) v,\quad c_{w,w}(\xn)=
\ii_{w,\infty}(\xn)^{-1}.\]

We expect that $M$ embeds
as a submodule of the $GL_n$ version 
of $H_*^T(\affsp_{n,kn})$, in
which the
fixed points only consist of positive permutations
\[M\subset \bigoplus_{\mn,\tau} \C[\xn]y^{\mn} \tau
=\bigoplus_{w\in \wpos} e_w,\quad 
w=(\tau_1+m_{\tau_1}n,...,\tau_n+m_{\tau_n}n).\]
A construction of this type was used in \cite{carlsson2018affine}
for instance, in which the authors exhibited an isomorphism
$DR_n\cong H_*(\affsp_{n,n+1})$ related to the ones
studied in \cite{oblomkov2014geometric}, and used it
to study the diagonal coinvariant algebra $DR_n$ 
as a module over $\C[\xn]$.
In another example, O. Kivinen 
showed that Haiman's alternant ideal 
$J_n\subset \C[\xn,\yn]$ in general
Lie type satisfies a suitable
version of the GKM relations, and therefore injects into the
equivariant Borel-Moore homology of the Grassmannian version
of $\affsp_{n,kn}$. In type A,
when combined with Haiman's results,
it follows that the map is an isomorphism when the
$y$-variables are inverted \cite{kivinen2020unramified}.

Now let 
\[b_{w,k}(\xn)=\prod_{i<j}(x_i-x_j)^{k-\#\edges^{i,j}_{kn}(w)},\]
whose degree is $\dimv_{kn}(w)$. The following conjecture
illustrates the connection with Theorem A in the case $k=1$:
\begin{conjecture}
\label{conj:free}
There exist free generators 
$A_w=\sum_{v} c_{v,w}(\xn) e_v \in M$ as $v,w\in \wpos_n$
satisfying the following properties:
\begin{enumerate}
    \item The $A_w$ freely generate $M$ as a $\C[\xn]$-module.
\item  The coefficients satisfy $c_{v,w}(\xn)=0$ unless $v\bruleq w$,
and the leading term is given by $c_{w,w}(\xn)=b_{w,1}(\xn)$.
\item For any compositions 
$\alpha,\beta$, if $w$ is the element of 
maximal length in $S_\alpha \backslash \wpos_n / S_\beta$,
then $A_w \in M^{\alpha,\beta}$, 
the invariant subspace with respect to the 
product of the corresponding Young subgroups.
\end{enumerate}
In particular, there is the expected freeness of 
$M$ over $\C[\xn]$.
\end{conjecture}

To connect this with Theorem A, observe that
combining \eqref{eq:nabcauchy}
with Conjecture \ref{conj:bkr}, we would have
\begin{equation}
\label{eq:nabmod}
\nabla e_n\left[\frac{XY}{(1-q)(1-t)}\right]=
\frob_{Y,X} M.\end{equation}
On the other hand, in light of Conjecture \ref{conj:free},
we expect that $\frob_{Y,X} M=\Omega_1[X,Y]$ with respect
to the dot and star actions mentioned in the introduction.
To see this in the case of the Hilbert series, 
we take the contribution to $\Omega_1[X,Y]$ 
for which $\an,\bn$ have all distinct entries. 
Then the automorphism factor is trivial, and we obtain
the sum in \eqref{eq:raths}, 
using Proposition \ref{prop:paff} to relate the corresponding 
$\dinv$ statistics.

\section{Main results}
We can now state and prove our main Theorem,
and some consequences.
\subsection{Main theorem}
\label{sec:mainthms}
Recall the conventions for summations of sorted representatives
described in Section \ref{sec:combex}. We have our main theorem:
\begin{thm}
\label{thm:thm1}
For any $k\geq 1$, we have
\begin{equation}\label{eq:nablahn}
\nabla^k e_n\left[\frac{XY}{(1-q)(1-t)}\right] 
= \sum_{[\mn, \an, \bn]}  \frac{t^{|\mn|} 
q^{\dinv_k(\mn, \an,\bn)}}{(1-q)^n \aut_q(\mn,\an,\bn)}
X_{\an} Y_{\bn}.
\end{equation}
\end{thm}
Before proving Theorem \ref{thm:thm1}, we state a few
immediate consequences. Let
\begin{equation}
\label{eq:csfa}
\csfa_{\dyckpath}[Y;q]=
\sum_{\bn} q^{\inv_{\dyckpath}(\bn)}Y_{\bn},
\end{equation}
where $\inv_{\dyckpath}(\bn)$ is defined in Section
\ref{sec:comb}.
\begin{prop}
\label{prop:nabcsf}    
The right hand side of \eqref{eq:nablahn} is given by
\[\sum_{[\mn,\an]} 
\frac{t^{|\mn|} q^{\dinv_k(\mn, \an)}}{(1-q)^n \aut_q(\mn,\an)}
X_{\an} \csfa_{\dyckpath_k(\mn,\an)}[Y;q].\]
\end{prop}
\begin{proof}
Notice that whenever $m_i=m_j$, $a_i=a_j$, and $b_i<b_j$,
switching the order of $b_i$ and $b_j$ always increases
$\dinv_k(\mn,\an,\bn)$ 
by one. Therefore we may remove the
sorting condition in $\bn$, i.e. replace $[\mn,\an,\bn]$
with $[\mn,\an],\bn$ in the sum, and remove it from the
automorphism factor without changing the answer. Then we have
\[\dinv_k(\mn,\an,\bn)=\dinv_k(\mn,\an)+
\inv_{\dyckpath_k(\mn,\an)}(\bn).\]
\end{proof}
We have the following interpretation of
$\xi_{\dyckpath}[Y;q]$. Let
\begin{equation}
\label{eq:csf}
\csf_{\dyckpath}[Y;q]=
\sum_{\bn:(i,j)\in \uio(\dyckpath)\Rightarrow b_i\neq b_j} 
q^{\inv_\dyckpath(\bn)}Y_{\bn}
\end{equation}
be Stanley's chromatic symmetric function.
\begin{prop}
\label{prop:xitocsf}
We have that
\begin{equation} \label{eq:xitocsf}
\csfa_{\dyckpath}[Y;q]=(1-q)^n 
\omega\csf_{\dyckpath}[Y(1-q)^{-1};q].
\end{equation}
In particular, it is a symmetric function.
\end{prop}
\begin{proof}
We have that $\xi_{\dyckpath}[Y;q]$ is the same as the
LLT polynomial $\chi_{\dyckpath}[Y;q]$ in 
\cite{carlsson2018proof}, and the statement
follows from Proposition 3.5 of that paper.
\end{proof}
\begin{rem}
This can also be seen 
using a conjecture of Shareshian and Wachs
\cite{shareshian2012chromatic}, 
later proved in two different ways in
\cite{brosnan2018hessenberg,guay2016second}, 
which would show that both sides of \eqref{eq:xitocsf}
are equal to the Frobenius character
of the equivariant cohomology 
of the regular semisimple Hessenberg variety.
This should have a geometric interpretation, and we expect
that it corresponds to the
paving of the affine Springer fiber by Hessenberg varieties from
\cite{goresky2003purity}.
\end{rem}
As a corollary, we have the
expression for $(\nabla^k e_1^n,e_n)$ from
\cite{elias2016computation}.
It was later proved in \cite{gorsky2017hilbert}, where it
was shown that both sides equal the
Poincar\'{e} polynomial for the Khovanov-Rozansky knot homology
of $k$th power of the full twist, and also the Hilbert series for the
$k$th power of Haiman's alternant ideal $J_n^k$.
\begin{cor}
We have that
\begin{equation}
    \label{eq:fulltwist}
(\nabla^k e_1^n,e_n)=\frac{1}{(1-q)^n}\sum_{\mn} 
t^{|\mn|}q^{d_k(\mn)},
\end{equation}
where
\[d_k(\mn)=\sum_{i<j} \max(k-m_i+m_j+1,k-m_j+m_i).\]
\end{cor}
\begin{proof}
The left hand side of \eqref{eq:fulltwist}
is the result of taking the coefficient
of $X_{\an}$ and $Y_{\bn}$ 
for $\an=(1,...,n)$ and $\bn=(1,...,1)$ in \eqref{eq:nablahn}.
Since the entries of $\an$ are distinct, there is no automorphism
factor. Now notice that the compositions $\mn$
are in bijection with sorted pairs $[\mn',\an]$ where $\an$
has distinct elements, meaning it is a permutation, and that
$d_k(\mn)=\dinv_k(\mn',\an)$.
\end{proof}
\subsection{Proof of Theorem A}
\label{sec:proofa}
We give our first proof of Theorem
\ref{thm:thm1} by taking
equation \eqref{eq:nablahn}
 as the definition of an operator
on symmetric functions, 
 and verifying that it satisfies
 the defining properties of $\nabla^k$, similar
 to the approach 
 in \cite{haglund2005combinatorial}.
Let us define an operator $\nablat{k}$ on symmetric
functions by
\begin{equation}
\label{eq:nablatdef}    
\nablat{k} e_n\left[\frac{XY}{(1-q)(1-t)}\right]=
\Omega_k[X,Y],
\end{equation}
where $\Omega_k[X,Y]$ denotes the right
hand side of \eqref{eq:nablahn}.
By Propositions 
\ref{prop:nabcsf} and 
\ref{prop:xitocsf}, and the symmetry
in $X$ and $Y$, the right
side of \eqref{eq:nablatdef} is symmetric
in both sets of variables, and so
this formula uniquely defines an operator acting
on the $X$-variables.

We will prove the following proposition.
\begin{prop}
    \label{prop:van}
We have the following properties of $\Omega_k[X,Y]$:
\begin{enumerate}
    \item \label{item:sym}
    It is symmetric 
    in the two sets of variables,
    $\Omega_k[X,Y]=\Omega_k[Y,X]$.
\item \label{item:van}
If $\lambda,\mu$
are partitions, then
the coefficient of 
$X^{\lambda} Y^{\mu}$ in 
$\Omega_k[X(t-1),Y(q-1)]$ is zero
unless $\lambda \dleq \mu'$
in the dominance order.
\item \label{item:lt} 
The leading coefficient in front of $X^{\lambda}Y^{\lambda'}$
is $q^{kn(\lambda')}t^{kn(\lambda)}$.
\end{enumerate}
\end{prop}
In particular, we have
\begin{cor}
\label{cor:van}
We have that $\nablat{k}=\nabla^k$, proving
Theorem \ref{thm:thm1}.
\end{cor}
\begin{proof}
We have already shown that $\Omega_k[X,Y]$ is
a symmetric function in the $Y$-variables
in Proposition \ref{prop:xitocsf}, and by the
$X\leftrightarrow Y$ symmetry it is symmetric in both.
The three properties in Proposition \ref{prop:van}
then correspond to the defining properties of 
$\nabla^k$, namely self-adjointness
in the modified Macdonald inner product, triangularity,
and the correct leading term.
\end{proof}
The $X\leftrightarrow Y$ symmetry item is clear from the
symmetry in the definition of $\dinv_k$.
We now turn to the hard part which is
showing triangularity. We first evaluate the plethystic substitution
$Y\mapsto Y(q-1)$ using
\eqref{eq:csf}
and \eqref{eq:xitocsf} to obtain
\begin{equation}
    \label{eq:nabhntv}
\Omega_k[X,Y(q-1)]=(-1)^n
\sum_{\substack{[\mn,\an,\bn] \\
i \text{ $k$-attacks } j\Rightarrow b_i\neq b_j}}  t^{|\mn|}q^{\dinv_k(\mn,\an,\bn)}    
X_{\an}Y_{\bn},
\end{equation}
noting that the automorphism
factors all disappear because we never have
a nonzero term with 
$(m_i,a_i,b_i)=(m_j,a_j,b_j)$ for $i\neq j$.

We would now like to evaluate the substitution
$X\mapsto X(t-1)$. To do this, we must
write equation \eqref{eq:nabhntv} as
a quasi-symmetric function in the $X$-variables.
We first sort the triples in a different
order, so that $[\an,\mn,\bn]'$ is a triple
in which the $a_i$ are in descending order,
the $m_i$ are in increasing order to break
ties, and the $b_i$ are in descending order
to breaking ties. This is the reverse of
the usual order, modified so that $\an$ has
priority over $\mn$.
Define two conditions $\mathrm{W}$ (wrong) and $\mathrm{NW}$ (not wrong) on pairs $\mn,\bn$:
\[\begin{array}{rl}
\mathrm{W}_i: & \text{if}\;m_{i}>m_{i+1}  \;\text{or}\; m_{i}=m_{i+1},\,b_{i}<b_{i+1},\\
\mathrm{NW}_i: & \text{if}\; m_{i}<m_{i+1} \;\text{or}\; m_{i}=m_{i+1},\,b_{i}> b_{i+1}.
\end{array}
\]
We can reconstruct the condition
of when a nonzero
term in \eqref{eq:nabhntv}
must have the inequality $a_i>a_{i+1}$
based on the ordering
$\mn,\bn$, to produce a quasi-symmetric expansion.

For $(\an,\mn,\bn)=[\an,\mn,\bn]'$ 
reverse sorted, we have that
\[\dinv_k([\mn,\an,\bn])=
\vdinv_k(\mn,\bn)=
\sum_{i<j}
\max(d^{i,j}_k(\mn,\bn),0)\]
for every nonzero term in \eqref{eq:nabhntv},
where
\begin{equation}
    \label{eq:vdinvdef}
d_k^{i,j}(\mn,\bn)=
\begin{cases}
k+m_j-m_i+\ddel(b_i>b_j) & m_i>m_j \\
k-1+m_i-m_j+\ddel(b_i<b_j) & m_i\leq m_j
\end{cases}
\end{equation}
and $\ddel$ is the delta function, 1 for
true, 0 for false. 
For any $\mn$ we will say that
$i$ $k$-attacks $j$ if
\[m_i \notin \{m_j-k+1,...,m_j+k\},\]
which is the same as $i'$ and $j'$
attacking each other in $[\mn,\an]$
if $\mn$ is the representative in which
$[\an,\mn]'$ are sorted.

We can now write
\begin{equation}
    \label{eq:hntquasi1}
\Omega_k[X,(q-1)Y]
= (-1)^n\sum_{\substack{\mn, \bn: \\
i \text{ $k$-attacks } j\Rightarrow
b_i\neq b_j}} t^{|\mn|} q^{\vdinv_k(\mn, \bn)} Y_{\bn} \sum_{\substack{a_1\geq \cdots \geq a_n:\\ \mathrm{W}_i\;\Rightarrow\; a_i>a_{i+1}}} X_{\an},
\end{equation}
which is a quasi-symmetric expansion.
We may apply the operator of the
substitution $F[X]\mapsto F[(t-1)X]$
using the standardization approach
from \cite{haglund2005combinatorial} to obtain
\[\Omega_k[(t-1)X,(q-1)Y]
 = \sum_{\substack{\mn, \bn: \\
i \text{ $k$-attacks } j\Rightarrow
b_i\neq b_j}} t^{|\mn|}  q^{\vdinv_k(\mn, \bn)} Y_{\bn}\times\]
\[ \sum_{l=0}^n 
(-t)^l
\sum_{\substack{a_1\leq \cdots \leq a_{n-l}:\\ \mathrm{NW}_i\;\Rightarrow\; a_i<a_{i+1}}} \sum_{\substack{a_{n-l+1}\geq \cdots \geq a_{n}:\\ \mathrm{W}_i\;\Rightarrow\; a_i>a_{i+1}}} X_{\an}=\]
\begin{equation}\label{eq:nabla vanishing1}
 \sum_{\substack{\mn, \bn: \\
i \text{ $k$-attacks } j\Rightarrow
b_i\neq b_j}} t^{|\mn|}  q^{\dinv_k(\mn, \bn)} Y_{\bn} \sum_{l=0}^n (-t)^{l} \sum_{\substack{a_1\leq \cdots \leq a_{n-l}:\\ a_i=a_{i+1}\;\Rightarrow\;\mathrm{W}_i}} \sum_{\substack{a_{n-l+1}\geq \cdots \geq a_{n}:\\ a_i=a_{i+1}\;\Rightarrow\;\mathrm{NW}_i}} X_{\an}.
\end{equation}

Before proving the vanishing, it will be helpful
to write equation \eqref{eq:nabla vanishing1}
in a more convenient form by collecting powers of $t$.
Define the rotation operator $\rho$ on pairs
$(\mn,\bn)$ by
$\rho(\mn, \bn) = (\mn', \bn')$, where 
\begin{equation}
    \label{eq:rotation}
m'_i=m_{i-1},\quad b'_i=b_{i-1},\quad m'_1=m_n+1, \quad b'_1=b_n,
\end{equation}
which satisfies 
\[
\vdinv_k(\rho(\mn, \bn))=\vdinv_k(\mn, \bn),\quad \area(\rho(\mn,\bn))=\area(\mn, \bn)+1,
\]
where $\area(\mn,\bn)=|\mn|$. 
Moreover, for $1\leq i<n-1$, $\mathrm{W}_i$ for $(\mn, \bn)$ is equivalent to $\mathrm{W}_{i+1}$ for $\rho(\mn, \bn)$. 
The triples in \eqref{eq:nabla vanishing1} are 
then bijectively mapped via $\rho^l$ to triples satisfying $m_1,\ldots,m_l\geq 1$, so the
right hand side of \eqref{eq:nabla vanishing1} becomes
\begin{equation}
\label{eq:nabla vanishing 3}
\Omega_k[X(t-1),Y(q-1)]=
\sum_{(l,\an,\mn,\bn)\in \vanset(n,k)}
(-1)^{l}t^{|\mn|}q^{\vdinv_k(\mn,\bn)}
X_{\an} Y_{\bn},
\end{equation}
and the summation set is given in the following definition:
\begin{defn}
Let $\vanset(n,k)$ be the set of all
quadruples $(l,\an,\mn,\bn)$ satisfying
\begin{enumerate}
\item \label{item:ank sorted} The terms $(a_i,m_i,b_i)$ are 
sorted for $l+1 \leq i \leq n$, and in
reverse order
for $1\leq i \leq l$.
    \item \label{item:ank nonzero} $m_i>0$ for $1 \leq i\leq l$.
\item \label{item:ank attacks} If $i$ $k$-attacks $j$, i.e.
$m_i \in \{m_j-k+1,...,m_j+k\}$ 
for $i<j$, then $b_i\neq b_j$.
\end{enumerate}
\label{def:vanset}
\end{defn}
\begin{example}
\label{example:ank}
For instance, we would have
\[A=\left(\begin{array}{c|cc|cccc}
\an&3&2&1&1&2&4\\
\mn&3&1&0&0&0&0\\ 
\bn&1&3&2&4&1&5\end{array}\right)\in \vanset(6,2),\]
where we are drawing a dividing line to indicate that $l=2$.
Below is the table of the contributions (before taking the max with zero)
to $\vdinv_2(\mn,\bn)$:
\[(\vdinv_2^{i,j}(\mn,\bn)) =\left(\begin {array}{rrrrrr} 1&\phantom{-}0&-1&-1&-1&-1\\ -1&1
&2&1&2&1\\ -2&1&1&2&1&2\\ -2&0&1&1&
1&2\\ -2&1&2&2&1&2\\ -2&0&1&1&1&1
\end {array} \right)\]
We see that $\vdinv_2(\mn,\bn)=16$, by adding up the positive entries
above the diagonal.
\end{example}
We now demonstrate the triangularity of equation
\eqref{eq:nabla vanishing 3} by finding an involution
$\invo_k:\vanset(n,k)\rightarrow \vanset(n,k)$ which sends
a quadruple $(l,\an,\mn,\bn)$ to itself, or sends it to
one which cancels it in \eqref{eq:nabla vanishing 3}.
We then show that the set of fixed points are empty
unless the dominance order property is satisfied.
\begin{defn}
\label{def:move}
For any $i$ and any quadruple $(l,\an,\mn,\bn)$, we define
\[\move_i(l,\an,\mn,\bn)=(l',\an',\mn',\bn')\]
where $l'=l-1$ if $i\leq l$ or $l+1$ if $i>l$, and
$(\an',\mn',\bn')$ is the result of inserting $(a_i,m_i,b_i)$
in the unique position on the opposite side of the dividing line
$l$ so that $(l',\an',\mn',\bn')$ is sorted as in 
Condition \ref{item:ank sorted} of Definition \ref{def:vanset}.
\end{defn}
Notice that for any element of $\vanset(n,k)$,
we never have $(a_i,m_i,b_i)=(a_j,m_j,b_j)$ unless $i=j$,
because of condition \ref{item:ank attacks}. We therefore
have a unique permutation $\sigma$
so that $(\an_{\sigma},\mn_{\sigma},\bn_{\sigma})$
is overall sorted, not in reverse order for $i\leq l$.
\begin{defn}
Given $A=(l,\an,\mn,\bn)\in \vanset(n,k)$,
we will say that $i$ is $k$-\emph{movable} if
$\move_i(A)\in \vanset(n,k)$, and
for any $j$ with $\sigma_j<\sigma_i$, we have
$d_k^{i,j}(\mn,\bn),d_k^{j,i}(\mn,\bn)\leq 0$.
\label{def:invo}
Let $\invo_k : \vanset(n,k)\rightarrow \vanset(n,k)$ be the involution
defined by setting
\[\invo_k(A)=\begin{cases}
A & \mbox{no element $i$ is $k$-movable},\\
\move_i(A) & \mbox{$i$ the movable element with smallest
value of $\sigma_i$.}\end{cases}\]
\end{defn}
Then if $A'=\invo_k(A)$ is not a fixed point,
the criteria for being $k$-movable ensures that
\begin{equation}
    \label{eq:dinvo}
    \vdinv_k(\mn',\bn')=\vdinv_k(\mn,\bn),\quad
|\mn'|=|\mn|,\quad
(-1)^l=-(-1)^{l'},
\end{equation}
which turns out to be enough
to cancel terms in \eqref{eq:nabla vanishing 3}.
\begin{example}
\label{example:invo}
Let us compute the involution on the term $A\in \vanset(6,2)$
from Example \ref{example:ank}.
We have that $\sigma=(5,3,1,2,4,6)$.
The smallest element is therefore
$i=3$, which is not moveable because $m_3=0$, so we cannot move it to the left of the dividing line $l=2$ without violating condition \ref{item:ank nonzero} of the definition of $\vanset(n,k)$.
The next smallest values of $i=4,2,5$ cannot be moved because we have
$d_2^{i,j}(\mn,\bn)>0$ or $d_2^{j,i}(\mn,\bn)>0$ 
for some $j$ earlier in the list.
However, $i=1$ is moveable, and we end up with
\[\invo_k(A)=
\left(\begin{array}{c|c|cccccc}
\an& 2& 1&1&2&3&4\\ 
\mn&1&0&0&0&3&0\\
\bn& 3&2&4&1&1&5
\end{array}\right).\]
\end{example}

Let $\vanset'(n,k)$ denote the fixed points of $\invo_k$.
\begin{prop}
\label{prop:dom}
If $\lambda,\mu$ are partitions,
then the set of all $(l,\an,\mn,\bn)\in \vanset'(n,k)$
for which $\alpha(\an),\alpha(\bn)=\lambda,\mu$
is empty unless $\lambda\dleq \mu'$ in the dominance order.
If $\lambda=\mu'$, then it contains a unique element.
\end{prop}
We will prove it through some lemmas. Let
$\pardiag(A)=\pardiag(\an,\mn,\bn)$ be the 
result of filling the $r$th row of the composition
$\alpha=\alpha(\an)$ with the
pairs $(m_i,b_i)$ for which $a_i=\alpha_r$, in such a way that 
the order is compatible with the reading order. For instance,
\begin{equation}
\label{eq:pardiag ex1}    
\pardiag\left(\begin{array}{c|ccccccc}
\an&2&2&1&2&1&3&1\\
\mn&1&0&0&1&2&1&0\\
\bn&1&1&3&2&3&1&1
\end{array}\right)=
\begin{ytableau}
23&01&03\\
11&12&01\\
11\\
\end{ytableau}
\end{equation}
Notice that the diagram does not depend on the ordering
or on $l$.
\begin{lem}
\label{lemma:tbound}
If $A\in \vanset'(n,k)$, then every $m$-number 
in row $r$ of $\pardiag(A)$ is at most $(r-1)k$.
\end{lem}
\begin{proof}
Since the rows of the diagram are decreasing,
it suffices to check the inequality for the first element of each
row. The first element of the first row must be zero,
otherwise the lowest element would be movable by
simply switching the position of the dividing line, which can only
violate Condition \ref{item:ank nonzero} of Definition
\ref{def:vanset}.
If some element is greater than $(r-1)k$,
then there must be a rows whose first entry exceeds
all previous entries by more than $k$, in which case
it is movable according
to \eqref{eq:vdinvdef}.
\end{proof}

\begin{lem}
\label{lemma:domb}
If $A\in \vanset'(n,k)$, then the
same $b$-number can appear in the first $r$ rows of 
$\pardiag(A)$ at most $r$ times. If it appears 
the maximum $r$ times, then they all occur in different
rows, and all occurrences $b_i$ are to the right of the 
dividing line, $i\geq l+1$.
\end{lem}
\begin{proof}
Suppose the number $b$ appears $r+1$ times in rows $1$ through $r$.
Let $x_1,...,x_r$ denote the set of the corresponding values of
$m_i$ in the order they appear in $\mn$, 
for instance $(1,0,1,0)$ for the $b$-value of $1$
in \eqref{eq:pardiag ex1}.
Then we must have that 
\begin{equation}
\label{eq:domb xcond}
x_s \leq x_t-k \mbox{ or } x_s \geq x_t+k+1
\end{equation}
for $1\leq s<t \leq r+1$
by condition \ref{item:ank attacks}
of Definition \ref{def:vanset}.
Now let $0\leq y_1\leq \cdots \leq y_{r+1}$ denote the
same set of numbers as the $x_s$ but in sorted order.
By \eqref{eq:domb xcond} we have that $y_{s+1}\geq y_s+k$,
so that $y_{r+1}\geq rk$, which contradicts Lemma \ref{lemma:tbound}.

To prove the second statement,
define $0\leq y_1\leq \cdots \leq y_r$ as above. 
Then by the same reasoning we have $y_r\geq (r-1)k$
and also $y_r\leq (r-1)k$ by the same lemma, so we must
have $y_s=sk$.
Then only the first case is
possible in \eqref{eq:domb xcond}, and so all the $x_s$
are same order $x_s=y_s$. 
Since $x_1=0$, it must
be to the right of the dividing line because of condition
\ref{item:ank nonzero}, and so the rest are as well.
Then if two $b$-values appear in the same row, there
will be increasing $m$-values for the same $a$-value,
so in the wrong order for the entries to the right of the 
dividing line.
\end{proof}

We now prove Proposition \ref{prop:dom}.
\begin{proof}
The shape of $\pardiag(l,\an,\mn,\bn)$ is just $\lambda$,
and the first statement of Lemma \ref{lemma:domb} easily shows 
that $\lambda\dleq \mu'$. 

For the second statement, if $\lambda=\mu'$, then the 
$i$th lowest $b$-number
appears $\lambda'_i$ times.
By the second statement of Lemma \ref{lemma:domb}, it appears
once in every row up to $\lambda'_i$, to the right of the dividing line, and (by the proof) with corresponding $m$-values $0,k,2k,...$
There is a unique term with these properties, namely
$l=0$ since all elements are to the right, and
\[\an=(1^{\lambda_1},...,l^{\lambda_r}),\ 
\mn=(0^{\lambda_1},...,(rk)^{\lambda_r}),\ 
\bn=(1,...,\lambda_1,...,1,...,\lambda_r).\]
\end{proof}

Finally, we can prove Proposition \ref{prop:van}, and therefore
Theorem \ref{thm:thm1}.
\begin{proof}
The first statement is clear from the symmetry of 
$\dinv_k$ in $\an$ and $\bn$. 
The second statement follows from Proposition
\ref{prop:dom}, since all terms in \eqref{eq:nabla vanishing 3}
corresponding to $A\in \vanset(n,k)-\vanset'(n,k)$
cancel with $\invo_k(A)$ by \eqref{eq:dinvo}.
Finally, the leading term from the proof
is easily seen to 
be $q^{kn(\lambda')}t^{kn(\lambda)}$.
\end{proof}

\subsection{A new proof of the shuffle theorem}

We now show how to recover the shuffle
theorem from Theorem \ref{thm:thm1}.

Notice that if we have $b_i=b_j$ for any $i\neq j$
in \eqref{eq:nabhntv}, then we cannot have $m_i=m_j$, 
or $m_i=m_j+1$ and $a_i>a_j$. 
By the first condition, we can uniquely
sort the orbits so that the $b_i$ 
are sorted in reverse order, 
$b_1\geq \cdots \geq b_n$, and if $b_i=b_{i+1}$ then $m_{i}<m_{i+1}$. 
Then the second condition says:
\[
b_{i}=b_{i+1}\Rightarrow m_{i+1}>m_i+1 \text{ or } m_{i+1}=m_i+1\text{ and } a_{i+1}\leq a_i.
\]
For a pair of sequences $(\mn,\an)$ and a position $i$ we
will define two conditions, ``parking function at $i$,''
and ``not parking function at $i$'', noting
that one is the negation of the other:
\[\begin{array}{rl}
\pfi_{k,i}: & \text{if}\;m_{i+1}\leq m_{i}+k-1 \;\text{or}\; m_{i+1}=m_i+k,\,a_{i+1}>a_i,\\
\npfi_{k,i}: & \text{if}\; m_{i+1}>m_i+k \;\text{or}\; m_{i+1}=m_i+k,\,a_{i+1}\leq a_i.
\end{array}\]

We can now write
\begin{equation}
    \label{eq:hntquasi}
\nabla^k h_n\left[\frac{XY}{1-t}\right] 
= \sum_{\mn, \an} X_{\an} t^{|\mn|}  q^{d_k(\mn, \an)} \sum_{\substack{b_1\geq \cdots \geq b_n:\\ \pfi_{k,i}\;\Rightarrow\; b_i>b_{i+1}}} Y_{\bn},
\end{equation}
where $d_k(\mn,\an)$ is the same one defined in \eqref{eq:vdinvdef}.
We would like to evaluate the substitution
$Y\rightarrow (1-t)Y$.
To do this, notice that \eqref{eq:hntquasi}
is a sum of quasi-symmetric functions in $Y$. 
We can therefore compute the substitution
using the standardization approach
from \cite{haglund2005combinatorial}. The result is
\begin{equation}\label{eq:nabla final}
\nabla^k h_n[XY] = \sum_{\mn,\an} t^{|\mn|}  q^{d_k(\mn, \an)} \sum_{l=0}^n (-t)^l X_{\an} \sum_{\substack{b_1\geq \cdots \geq b_{n-l}:\\ \pfi_{k,i}\;\Rightarrow\; b_i>b_{i+1}}} 
\sum_{\substack{b_{n-l+1}\leq \cdots \leq b_{n}:\\ \npfi_{k,i}\;\Rightarrow\; b_i<b_{i+1}}} Y_{\bn}.
\end{equation}
Finally, let us make the evaluation $Y=-1$ in 
\eqref{eq:nabla final} adding an extra sign $(-1)^n$,
which amounts to counting only the terms in 
which the quasi-symmetric functions have
strict inequalities. We obtain
\begin{equation}\label{eq:shuffle conjecture 1}
\nabla^k e_n[X] = \sum_{l=0}^n (-t)^l \sum_{\substack{\mn, \an:\\\pfi_{k,i} \;\text{for}\; 1\leq i\leq n-l-1,\\\npfi_{k,i} \;\text{for}\; n-l+1\leq i 
\leq n-1}} X_{\an} t^{|\mn|}  q^{d_k(\mn, \an)}.
\end{equation}

We would like to cancel certain terms in the right hand side
of \eqref{eq:shuffle conjecture 1}, this time
using the rotation operator $\rho$ 
defined in \eqref{eq:rotation}.
Let
\[\mathcal{A}(n,k)=\left\{(l,\mn,\an):
\pfi_{k,i} \text{ for $1\leq i\leq n-l-1$},
\right.\]
\[\left.
\npfi_{k,i} \text{ for $n-l+1\leq i \leq n-1$}\right\},\]
and notice that for $1\leq i<n-1$, we have
that $\pfi_{k,i}$ for $(\mn, \an)$ is equivalent to $\pfi_{k,i+1}$ for $\rho(\mn, \an)$. 
Consider those triples 
$(l,\mn,\an)\in \mathcal{A}(n,k)$ satisfying
\begin{enumerate}
	\item[(1A)] $l>0$,
	\item[(2A)] $PF_{k,1}$ for $\rho(\mn, \an)$ if $l<n$.
\end{enumerate}
The image of $\rho$ on these triples is the set of triples $(l,\mn,\an)$ satisfying
\begin{enumerate}
	\item[(1B)] $l<n$,
	\item[(2B)] $NPF_{k,n-1}$ for $\rho^{-1}(\mn, \an)$ if $l>0$,
	\item[(3B)] $m_1>0$.
\end{enumerate}
It is clear that these are the sets of all 
$A\in \mathcal{A}(n,k)$ so that 
$\rho(A)\in \mathcal{A}(n,k)$ for the first set,
or $\rho^{-1}(A) \in \mathcal{A}(n,k)$ for the second.

We can now check the following proposition,
which implies that the two sets have no elements in
common, and so the terms coming
from the two sets cancel each other out in
\eqref{eq:shuffle conjecture 1}.
\begin{prop}
\label{prop:rho}
	The set of triples satisfying (1A), (2A), (1B), (2B), (3B) is empty.
\end{prop}
We let $\mathcal{A}'(n,k)\subset \mathcal{A}(n,k)$ 
denote the subset of contributing terms, 
which are not in either set.
We can now give a new proof of the shuffle theorem \cite{haglund2005combinatorial,carlsson2018proof},
noting that the conditions of the summation 
in \eqref{eq:shuffle conjecture} mean that $\mn$ 
is the area sequence of a Dyck path, and $\an$ is a word parking function, see \cite{haglund2008catalan}.
\begin{thm}
\label{thm:shuffle}
\begin{equation}\label{eq:shuffle conjecture}
\nabla^k e_n[X] = \sum_{\substack{\mn, \an: PF_{k,i} \;\text{for all $i$},\; m_1=0}} X_{\an} t^{|\mn|}  q^{d_k(\mn, \an)}.
\end{equation}
\end{thm}
\begin{proof}
Using Proposition \ref{prop:rho} 
to cancel terms in \eqref{eq:shuffle conjecture 1}, 
the terms that remain
are the ones that fail to satisfy at least one out of (1A) and (2A), and also fail at least one of (1B), (2B), (3B). If a term does not satisfy (2A), it means $l<n$ and $\npfi_{k,1}$ holds for $\rho(\mn,\an)$. In particular, we have $0< m_n+k\leq m_1$, so the only property that can fail among (1B), (2B), (3B) is (2B). 
Therefore, $l>0$ and $\pfi_{k,n-1}$ holds for $\rho^{-1}(\mn,\an)$, which is equivalent to $\pfi_{k,1}$ for $\rho(\mn,\an)$, 
a contradiction. 
Then among (1A), (2A), the property (1A) is the one that fails,
so we have $l=0$. 
The only property among (1B), (2B), (3B) 
that can fail in the case $l=0$ is (3B), so we have $m_1=0$.
\end{proof}

\section{Parabolic bundles}
\label{sec:bundles}
In this Section we deduce Theorem \ref{thm:thm1} by counting parabolic bundles in two different ways. 

\subsection{Counting formula}
On the first side, we will need a result from \cite{mellit2017poincarea} 
for counting bundles on $\BP^1$ over a finite field.
 Let $q$ be a prime power, and let $\Bk$ be the finite field with $|\Bk|=q$ elements. Let $S=\{s_1,\ldots,s_k\}\subset \BP^1(\Bk)$ be a collection of rational points. Let $N$ be a big integer (this will correspond to the number of variables in each alphabet). We need $k$ alphabets $X_1,\ldots,X_k$. The variables in alphabet $X_i$ are denoted $x_{i,j}$ ($1\leq i\leq k$, $1\leq j\leq N$).

\begin{defn}\label{defn:parabolic_bundle}
	A \emph{parabolic bundle} is a pair $(\CE, \bF)$, where $\CE$ is a vector bundle on $\BP^1$ over $\Bk$, and $\bF=(F_{i,j})_{1\leq i\leq k,0\leq j\leq N}$ is a collection of vector spaces so that for each $i$ we have
	\[
	0=F_{i,0}\subseteq F_{i,1} \subseteq \cdots \subseteq F_{i,N-1} \subseteq F_{i,N}=\CE(s_i).
	\]
	An endomorphism of $(\CE, \bF)$ is an endomorphism of $\CE$ preserving each $F_{i,j}$. An endomorphism $\theta$ is nilpotent if $\theta^n=0$ for some $n$.
\end{defn}
Here $\CE(s_i)$ is the fiber of $\CE$ over $s_i$. If $\CE$ had rank $n$, then $\CE(s_i)$ is an $n$-dimensional vector space. 

Parabolic bundles have the following discrete invariants:
\begin{itemize}
	\item $\rank(\CE)$ = rank of $\CE$,
	\item $\deg(\CE)$ = degree of $\CE$,
	\item $r_{i,j}=\dim(F_{i,j}/F_{i,j-1})$.
\end{itemize}
Note that $r_{i,\bullet}$ is a composition of $n$ for each $i=1,\ldots,k$ (of length $N$ with zeros allowed). These invariants are packaged in the following \emph{weight}:
\[
\weight(\CE,\bF) = t^{\deg} \prod_{i=1}^k \prod_{j=1}^N x_{i,j}^{r_{i,j}}.
\]

It is well-known that over $\BP^1$ every vector bundle is a sum of line bundles, so we can write $\CE=O(m_1)\oplus\cdots\oplus O(m_n)$. We write $\CE\geq 0$ if all $m_i\geq 0$. The following formula has been proved in \cite{mellit2017poincarea}:
\begin{equation}\label{eq:main formula}
\Omega = \sum_{(\CE,\bF):\CE\geq 0} \frac{\weight(\CE,\bF)\cdot |\Nilp(\CE,\bF)|}{|\Aut(\CE,\bF)|} = \sum_{n=0}^\infty \sum_{\lambda\vdash n} \frac{\prod_{i=1}^k \tilde H_{\lambda}[X_i;q,t]}{\prod_{a,l} (q^{a+1}-t^l)(q^a-t^{l+1})}.
\end{equation}
The summation goes over the isomorphism classes of pairs $(\CE,\bF)$. Here $\Nilp(\CE,\bF)$ denotes the set of all nilpotent endomorphisms of $\CE$ which preserve $\bF$, and similarly $\Aut(\CE,\bF)$ is the set of automorphisms. Equivalently, one can sum over the isomorphism classes of triples $(\CE,\bF,\theta)$ where $\theta\in\Nilp(\CE,\bF)$, and each summand is $\frac{\weight(\CE,\bF)}{|\Aut(\CE,\bF,\theta)|}$:
\[
\Omega = \sum_{(\CE,\bF,\theta):\CE\geq 0,\,\theta\in\Nilp(\CE,\bF)} \frac{\weight(\CE,\bF)}{|\Aut(\CE,\bF,\theta)|}.
\]

We transform the summation as follows. Assume $s_1=0$ and $s_2=\infty$. Let us run the summation over $\CE$, $\bF^0:=\bF_{1,\bullet}$, $\bF^\infty:=\bF_{2,\bullet}$ and $\theta$ first, and then go over the possible flags $\bF_{3,\bullet}$,\ldots, $\bF_{k,\bullet}$. We obtain
\[
\Omega = \sum_{(\CE,\bF^0,\bF^\infty)} \frac{t^{\deg} \prod_{j=1}^N x_{1,j}^{r_{1,j}} x_{2,j}^{r_{2,j}}}{|\Aut(\CE,\bF^0,\bF^\infty)|}\sum_\theta \sum_{\bF_{3,\bullet},\dots,\bF_{k,\bullet}} \prod_{i=3}^k  \prod_{j=1}^N x_{i,j}^{r_{i,j}}.
\]
We can interpret triple $\CE,\bF^0,\bF^\infty$ as a parabolic bundle with only two marked points $0$, $\infty$, and $\theta$ as its endomorphism. For fixed $\CE, \theta$ the flag $\bF_{i,\bullet}$ runs over all flags in $\CE(s_i)$ preserved by the restriction $\theta(s_i)$ of the endomorphism $\theta$ to $\CE(s_i)$. In particular, the summation and the product can be interchanged: $\sum_{\bF_{3,\bullet},\dots,\bF_{k,\bullet}} \prod_{i=3}^k = \prod_{i=3}^k \sum_{\bF_{i,\bullet}}$. The contribution 
\[
\sum_{\bF_{i,\bullet}} \prod_{j=1}^N x_{i,j}^{r_{i,j}}
\]
clearly depends only on the Jordan form of $\theta(s_i)$ and is given by the corresponding Hall-Littlewood polynomial $\tilde H_{\type \theta(s_i)}[X_i;q,0]$ (see \cite{mellit2017poincarea} for details). The notation $\type \theta(s_i)$ stands for the partition whose conjugate specifies the sizes of the Jordan blocks of $\theta(s_i)$. We obtain the following
\begin{cor}\label{cor:main_identity2}
    The generating function $\Omega$ can be written as follows:
\[   
\Omega = \sum_{n=0}^\infty \sum_{\lambda\vdash n} \frac{\prod_{i=1}^k \tilde H_{\lambda}[X_i;q,t]}{\prod_{a,l} (q^{a+1}-t^l)(q^a-t^{l+1})}
 \]
 \[
 = \sum_{(\CE,\bF^0,\bF^\infty)} \frac{t^{\deg} \prod_{j=1}^N x_{1,j}^{r_{1,j}} x_{2,j}^{r_{2,j}}}{|\Aut(\CE,\bF^0,\bF^\infty)|}\sum_{\theta\in \Nilp(\CE,\bF^0,\bF^\infty)} \prod_{i=3}^k \tilde H_{\type \theta(s_i)}[X_i;q,0],
\]
where the first summation on the right hand side runs over the isomorphism classes of parabolic bundles with marked points $0,\infty$.
\end{cor}

Here we are interested in expressions of the form
\[
\nabla^k h_n\left[-\frac{XY}{(q-1)(t-1)}\right] = \sum_{\lambda\vdash n} \left(q^{n(\lambda')} t^{n(\lambda)}\right)^k \frac{H_{\lambda}[X;q,t]H_{\lambda}[Y;q,t]}{\prod_{a,l} (q^{a+1}-t^l)(q^a-t^{l+1})}.
\]
Recall that
\[
(H_{\lambda}[X;q,t], s_{1^n}) = q^{n(\lambda')} t^{n(\lambda)},
\]
and by setting $t=0$
\[
(H_{\lambda}[X;q,0], s_{1^n}) = \begin{cases} q^{\binom{n}{2}} & \lambda=(n),\\
0 & \text{otherwise}.\end{cases}
\]
Applying $(-,s_{1^n})$ in the alphabets $X_3,\ldots,X_k$ to both sides of Corollary \ref{cor:main_identity2} and replacing $k$ by $k+2$ and relabeling $s_i$ we obtain
\begin{cor}\label{cor:main_identity3} Let $k\geq 0$, and let $\{s_1,\ldots,s_k\}$ be an arbitrary set of points on $\BP(\Bk)$ disjoint from $0$ and $\infty$. We have
\[
q^{k \binom{n}{2}} \sum_{\substack{(\CE,\bF^0,\bF^\infty)\\ \rank(\CE)=n}} \frac{t^{\deg} \prod_{j=1}^N x_j^{r_{1,j}} y_{j}^{r_{2,j}}}{|\Aut(\CE,\bF^0,\bF^\infty)|}|\Nilp_k(\CE,\bF^0,\bF^\infty)| =  \nabla^k h_n\left[-\frac{XY}{(q-1)(t-1)}\right],
\]
where $\Nilp_k(\CE,\bF^0,\bF^\infty)$ denotes the set of nilpotent endomorphisms $\theta$ satisfying $\theta(s_i)=0$ for $i=1,\ldots,k$.
\end{cor}

\subsection{Parabolic bundles with two marked points}
Next we will use an explicit classification of triples $(\CE,\bF^0,\bF^\infty)$ to give an alternative formula for the generating function in Corollary \ref{cor:main_identity3}. The building blocks of the classification will be parabolic bundles of rank $1$, i.e. parabolic line bundles:

\begin{example}
Consider Definition \ref{defn:parabolic_bundle} in the case $\rank(\CE)=1$. Then $\CE(s_i)$ is a vector space of dimension $1$, so the sequence of vector spaces $0=F_{i,0}\subseteq\cdots\subseteq F_{i,N}=\CE(s_i)$ is determined by an integer $j_i$ such that $F_{i,j_i-1}=0$, $F_{i,j_i}\neq 0$. Since we are on $\BP^1$, the line bundle $\CE$ is uniquely determined by its degree $m$. So a parabolic line bundle is uniquely determined by an integer $m$ and a tuple $(j_1, j_2, \ldots, j_k)$, $1\leq j_i\leq N$. In the case $k=2$, we will denote $a=j_1$, $b=j_2$. The corresponding parabolic line bundle is denoted by $O(m;a,b)$.
\end{example}

\begin{prop}\label{prop:decomposition}
    Let $(\CE,\bF^0,\bF^\infty)$ be a parabolic vector bundle of rank $n$ on $\BP^1$ with two marked points. There exists a unique multiset of triples $(m_1,a_1,b_1)$, \ldots, $(m_n,a_n,b_n)$ such that
    \[
    (\CE,\bF^0,\bF^\infty) = \bigoplus_{i=1}^n O(m_n;a_n,b_n).
    \]
\end{prop}

This can be thought of as a generalization of the classical Bruhat decomposition for $GL_n$.
There is a tedious direct proof based on several applications of the standard Bruhat decomposition, but we will use homological algebra instead. The proof will occupy the rest of the section.

Of course, the category of parabolic bundles is not an abelian category, but it can be embedded as a full subcategory into the abelian category of parabolic coherent sheaves, which has global dimension $1$, so all $\Ext^i$ vanish for $i>1$. The Euler form is given by (see \cite{mellit2017poincarea})
\begin{equation}\label{eq:homs}
\dim \Hom(\overline\CE,\overline\CE') - \dim \Ext(\overline\CE,\overline\CE') = \rank\CE \rank\CE'+ \rank\CE \deg \CE' - \rank\CE' \deg \CE
\end{equation}
\[
 - \sum_{i=1}^k \sum_{1\leq j<j'\leq N} r_{i,j}(\overline\CE) r_{i,j'}(\overline\CE').
\]
We denote by $\overline\CE$ the pair $(\CE,\bF)$. In the case $k=2$ we write $(\CE, \bF^0, \bF^\infty)$.

The dimension of $\Hom$ between two parabolic line bundles is given by
\[
\dim \Hom(O(m;j_1,\ldots,j_k), O(m';j_1',\ldots,j_k')) =\max(1+m'-m-\#\{i:j_i<j_i'\}, 0).
\]
By the formula for the Euler form we obtain
\[
\dim \Ext(O(m;j_1,\ldots,j_k), O(m';j_1',\ldots,j_k')) 
=\max(m-m'-1+\#\{i:j_i< j_i'\}, 0).
\]

Introduce a total order on parabolic line bundles in such a way that $O(m;j_1,\ldots,j_k)<O(m';j_1',\ldots,j_k')$ precisely when 
\[
(m,-j_1,\ldots,-j_k)<(m',-j_1',\ldots,-j_k') \;\text{lexicographically}.
	\]
This order clearly satisfies
\begin{prop}\label{prop:hom order}
    For two parabolic line bundles $L, L'$ if $\Hom(L, L')\neq 0$, then $L\leq L'$.
\end{prop}	

\begin{proof}[Proof of Proposition \ref{prop:decomposition}]
    Let $k=2$.	We prove the existence first. The proof goes by induction on the rank $n$. The case $n=1$ is clear. Assume $n>1$ and suppose $\ol\CE$ is a parabolic bundle of rank $n$. Let $L\subset \ol\CE$ be the maximal in our order parabolic line subbundle. From the theory of parabolic coherent sheaves, we have a short exact sequence
	\[
	0\to L \to \ol\CE \to \ol\CE'\to 0
	\]
	where $\ol\CE'$ is a parabolic bundle of rank $n-1$.
	By the induction assumption, $\ol\CE'\cong \bigoplus_{l=1}^{n-1} L_l$. Suppose the short exact sequence does not split. Then there exist $l$ such that for $L'=L_l$ we have
	\[
	\Ext(L', L)\neq 0 \;\Rightarrow\; m'-m-1+\#\{i:j_i'<j_i\} \geq 1.
	\]
	Note that since $k=2$ this implies $m'\geq m$.
	Our plan is to construct a bundle $L''$ such that
	\begin{enumerate}
	    \item $\Hom(L'', L')\neq 0$,
	    \item $\Ext(L'', L)=0$,
	    \item $L'' > L$.
	\end{enumerate}
	By the exact sequence
	\[
	\Hom(L'', \ol\CE) \to \Hom(L'', \ol\CE') \to \Ext(L'', L)
	\]
	these conditions would guarantee that any non-zero homomorphism $h\in\Hom(L'', L')\subset \Hom(L'', \CE')$ can be lifted to a non-zero homomorphism $L''\to \ol\CE$, and we would obtain a contradiction with the maximality of $L$.
	
	If $m'\geq m+1$, we pick $L''=O(m+1;N,N)$. This guarantees that $\dim \Hom(L'', L') = m'-m>0$, $\dim \Ext(L'', L) = 0$ and  $L''>L$, so the required conditions are satisfied. 
	
	Otherwise, we must have $m'=m$, $j_1'<j_1$ and $j_2'<j_2$. Picking $L'' = O(m;j_1',j_2)$ (or $O(m;j_1,j_2')$) satisfies $\dim \Hom(L'', L') = 1$, $\dim \Ext(L'', L) = 0$ and $L''> L$. So the existence have been proven.
	
	Note that we have in particular demonstrated that the the maximal line subbundle is a direct summand. By  Proposition \ref{prop:hom order} it must be present in any direct sum decomposition, and by successively splitting away the maximal subbundle we deduce the uniqueness.
\end{proof}

\begin{rem}
For $k>2$ the statement does not hold. For a counter-example for $k=3$, pick trivial bundle of rank $2$ and three lines in general position over the marked points.
\end{rem}

\subsection{Computations}
We are ready to identify all the ingredients in the left hand side of Corollary \ref{cor:main_identity3}. By Proposition \ref{prop:decomposition}, the summation runs over the set of sorted triples $[\mn, \an, \bn]$ (see Section \ref{sec:comb} for combinatorial notations). Each sorted triple corresponds to a direct sum of line bundles $L_i=O(m_i;a_i,b_i)$, which satisfy $L_1\geq \cdots \geq L_n$. Denote 
\[
O(\mn;\an,\bn) = \bigoplus_{i=1}^n O(m_i;a_i,b_i).
\]

\begin{prop}\label{prop:automorphisms}
    Suppose $(\mn,\an,\bn)$ is sorted. The number of automorphisms of $O(\mn;\an,\bn)$ is given by
    \[
    |\Aut(O(\mn;\an,\bn)| = (q-1)^n \aut_q(\mn,\an,\bn) q^{\sum_{i<j} \max(1+m_i-m_j - \delta_{a_j<a_i} - \delta_{b_j<b_i}, 0)}.
    \]
\end{prop}
\begin{proof}
    By Proposition \ref{prop:hom order}, automorphisms are given by block-upper-triangular matrices with block sizes equal to the multiplicities of triples $(m_i,a_i,b_i)$. A block-upper-triangular matrix is invertible precisely if the blocks are. So we obtain that the number of automorphisms is given by
    \[
    |\Aut(O(\mn;\an,\bn)| = \prod_{i<j:\; L_i\neq L_j} q^{\dim \Hom(L_j, L_i)} \times \prod_{i} |GL_{\mu_i}(\Bk)|,
    \]
    where $\mu_1, \mu_2, \ldots$ denote the multiplicities, $\sum_i \mu_i = n$.
    The number of elements of $GL_r(\Bk)$ is given by
    \[
    |GL_r(\Bk)| = q^{\binom{r}{2}} (q-1)^r [r]_q!.
    \]
    Since $\dim \Hom(L_i, L_i)=1$, we have
    \[
    |\Aut(O(\mn;\an,\bn)| = \prod_{i<j} q^{\dim \Hom(L_j, L_i)} (q-1)^n \aut_q(\mn,\an,\bn),
    \]
    and \eqref{eq:homs} gives the formula.
\end{proof}

The following statement is useful
\begin{prop}
    Suppose $\theta$ is an endomorphism of $\ol \CE$ which vanishes in some point $s\in\BP(\Bk)\setminus\{0,\infty\}$. Then $\theta$ is nilpotent.
\end{prop}
\begin{proof}
	For every $r>0$, $\trace \theta^r$ is a global function on $\BP^1$, therefore it is constant. Since it is zero at $s$, it must be zero everywhere. So $\trace \theta^r=0$ for all $r>0$, hence $\theta$ is nilpotent.
\end{proof}

Below we include the case $k=0$ for completeness. We have
\begin{prop}
    Suppose $(\mn,\an,\bn)$ is sorted and $k\geq 0$. We have
    \[
    |\Nilp_k(O(\mn;\an,\bn)| = q^{\sum_{i<j} \max(1-k+m_i-m_j - \delta_{a_j<a_i} - \delta_{b_j<b_i}, 0)} \times
    \begin{cases}
    1 & (k>0)\\
    q^{\sum_{i} \binom{\mu_i}{2}} & (k=0)
    \end{cases},
    \]
    where $\mu=(\mu_1,\ldots,\mu_l)$ are the multiplicities of the triples $(m_i,a_i,b_i)$.
\end{prop}
\begin{proof}
    As in the proof of Proposition \ref{prop:automorphisms}, the endomorphisms are given by block-upper-triangular matrices. In the case $k>0$ the blocks are automatically zero. The space of off-diagonal entries in position $(i,j)$ is given by $\Hom(L_j, L_i)$ ($i<j$). This is the space of polynomials of bounded degree. When polynomials are forced to have zeroes at $k$ further points, the dimension drops down by $k$. This completely describes the case $k>0$. For the case $k=0$ we need to count the number of nilpotent matrices in each block. This is given by $q^{r^2-r}$ for a block of size $r\times r$. For each block, the factor $q^{\sum_{i<j} \cdots}$ contains $q^{\binom{r}2}$, so extra factor $q^{\binom{r}2}$ has to be added.
\end{proof}

The remaining pieces of the left hand side of Corollary \ref{cor:main_identity3} are identified as follows:
\[
t^{\deg} = t^{|\mn|},\qquad \prod_{j=1}^N x_j^{r_{1,j}} = X_\an, \qquad \prod_{j=1}^N y_j^{r_{2,j}} = Y_\bn.
\]

\begin{example}
Let $k=0$. The left hand side of Corollary \ref{cor:main_identity3} becomes
\[
\sum_{[\mn,\an,\bn]} \frac{t^{|\mn|} X_\an Y_\bn q^{\sum_{i} \binom{\mu_i}{2}}}{(q-1)^n \aut_q(\mn,\an,\bn)}.
\]
So in each summand each triple $(m,a,b)$ with multiplicity $\mu$ contributes a factor of
\[
\frac{t^{\mu m} x_a^\mu y_b^\mu q^{\binom{\mu}{2}}}{(q-1)^\mu [\mu]_q!}.
\]
Summing over all $n$ we obtain every possible triple with every multiplicity, so the result can be written as an infinite product
\[
\prod_{m=0}^\infty \prod_{a,b=1}^N \sum_{\mu=0}^\infty \frac{t^{\mu m} x_a^\mu y_b^\mu q^{\binom{\mu}{2}}}{(q-1)^\mu [\mu]_q!} = \prod_{m=0}^\infty \prod_{a,b=1}^N \prod_{r=0}^\infty (1-x_a y_b t^m q^r) = \pExp\left[-\frac{X Y}{(1-t)(1-q)}\right],
\]
which matches the right hand side of Corollary \ref{cor:main_identity3}.
\end{example}

Our main conclusion is
\begin{thm}
For $k\geq 1$ we have
\[
\nabla^k h_n\left[-\frac{XY}{(q-1)(t-1)}\right] = \sum_{[\mn,\an,\bn]} \frac{t^{|\mn|} q^{\dinv_k(\mn,\an,\bn)} X_\an Y_\bn}{(q-1)^n \aut_q(\mn,\an,\bn)}.
\]
\end{thm}
\begin{proof}
    In view of the above computations and Corollary \ref{cor:main_identity3}, it remains to match the $q$-degree in each summand with $\dinv_k$. The $q$-degree is given by
    \[
    k \binom{n}{2} + \sum_{i<j} \max(1-k+m_i-m_j - \delta_{a_j<a_i} - \delta_{b_j<b_i}, 0) - \sum_{i<j} \max(1+m_i-m_j - \delta_{a_j<a_i} - \delta_{b_j<b_i}, 0).
    \]
    For each pair $i<j$ let $c_{i,j} = m_i-m_j - \delta_{a_j<a_i} - \delta_{b_j<b_i}$ and note that $c_{i,j}\geq -1$. Then the above sum can be written as
    \[
    \sum_{i<j} \left(k+\max(1-k+c_{i,j},0) - (1+c_{i,j})\right) = \sum_{i<j} \max(k-1-c_{i,j}, 0).
    \]
    Each summand matches the corresponding summand in Definition \ref{defn:dinv}
    \[
    \max(k-1-c_{i,j}, 0) = \dinv_k^{i,j}(\mn, \an, \bn),
    \]
    and therefore the $q$-degree equals $\dinv_k(\mn,\an,\bn)$.
\end{proof}

\bibliographystyle{amsalpha}
\bibliography{refs}

\newcommand{\etalchar}[1]{$^{#1}$}
\providecommand{\bysame}{\leavevmode\hbox to3em{\hrulefill}\thinspace}
\providecommand{\MR}{\relax\ifhmode\unskip\space\fi MR }
\providecommand{\MRhref}[2]{%
  \href{http://www.ams.org/mathscinet-getitem?mr=#1}{#2}
}
\providecommand{\href}[2]{#2}
\begin{thebibliography}{BGSLX16}

\bibitem[BC15]{brosnan2018hessenberg}
Patrick Brosnan and Timothy Chow, \emph{Unit interval orders and the dot action
  on the cohomology of regular semisimple {H}essenberg varieties}, Advances in
  Mathematics \textbf{329} (2015), 955--1001.

\bibitem[BG99]{bergeron1999science}
Francois Bergeron and Adriano~M. Garsia, \emph{Science fiction and
  {M}acdonald's polynomials}, Algebraic methods and {$q$}-special functions
  ({M}ontr\'eal, {QC}, 1996), CRM Proc. Lecture Notes, vol.~22, Amer. Math.
  Soc., Providence, RI, 1999, pp.~1--52. \MR{1726826}

\bibitem[BGHT99]{bergeron1999identities}
Francois Bergeron, Adriano~M. Garsia, Mark Haiman, and Glenn Tesler,
  \emph{Identities and positivity conjectures for some remarkable operators in
  the theory of symmetric functions}, Methods Appl. Anal. \textbf{6} (1999),
  no.~3, 363--420. \MR{1803316}

\bibitem[BGSLX16]{bergeron2016compositional}
Francois Bergeron, Adriano Garsia, Emily Sergel~Leven, and Guoce Xin,
  \emph{Compositional {$(km,kn)$}-shuffle conjectures}, International
  Mathematics Research Notices. IMRN (2016), no.~14, 4229--4270. \MR{3556418}

\bibitem[CM18]{carlsson2018proof}
Erik Carlsson and Anton Mellit, \emph{A proof of the shuffle conjecture}, J.
  Amer. Math. Soc. \textbf{31} (2018), no.~3, 661--697. \MR{3787405}

\bibitem[CO18]{carlsson2018affine}
Erik Carlsson and Alexei Oblomkov, \emph{Affine schubert calculus and double
  coinvariants}, submitted (2018), 1--36, arXiv:1801.09033.

\bibitem[EH16]{elias2016computation}
Ben Elias and Matthew Hogancamp, \emph{On the computation of torus link
  homology}, Compositio Mathematica \textbf{155} (2016), 164--205.

\bibitem[GH17]{gorsky2017hilbert}
E.~Gorsky and Matthew Hogancamp, \emph{Hilbert schemes and {$y$}-ification of
  {K}hovanov-{R}ozansky homology}, arXiv: Geometric Topology (2017), 61.

\bibitem[GKM03]{goresky2003purity}
M.~Goresky, Robert~E. Kottwitz, and R.~Macpherson, \emph{Purity of equivalued
  affine {S}pringer fibers}, Representation Theory of The American Mathematical
  Society \textbf{10} (2003), 130--146.

\bibitem[GKM04]{goresky2004unramified}
Mark Goresky, Robert Kottwitz, and Robert MacPherson, \emph{Homology of affine
  {S}pringer fibers in the unramified case}, Duke Math. J. \textbf{121} (2004),
  no.~3, 509--561.

\bibitem[GM13]{gorsky2013compactified}
Evgeny Gorsky and Mikhail Mazin, \emph{Compactified {J}acobians and {$q,
  t$}-{C}atalan numbers, {I}}, Journal of Combinatorial Theory, Series A
  \textbf{120} (2013), no.~1, 49--63.

\bibitem[GMV14]{gorsky2014affine}
Eugene Gorsky, Mikhail Mazin, and Monica Vazirani, \emph{Affine permutations
  and rational slope parking functions}, Transactions of the American
  Mathematical Society \textbf{368} (2014), 30.

\bibitem[GORS14]{gorsky2014torus}
Eugene Gorsky, A.~Oblomkov, J.~Rasmussen, and V.~Shende, \emph{Torus knots and
  the rational daha}, Duke Mathematical Journal (2014), no.~14, 2709--2794.

\bibitem[GP16]{guay2016second}
M.~Guay-Paquet, \emph{A second proof of the {S}hareshian-{W}achs conjecture, by
  way of a new hopf algebra}, arXiv: Combinatorics (2016), 1--36.

\bibitem[Hag08]{haglund2008catalan}
James Haglund, \emph{The {$q$},{$t$}-{C}atalan numbers and the space of
  diagonal harmonics}, University Lecture Series, vol.~41, American
  Mathematical Society, Providence, RI, 2008, With an appendix on the
  combinatorics of Macdonald polynomials. \MR{2371044 (2009f:05261)}

\bibitem[Hai01a]{haiman2001hilbert}
Mark Haiman, \emph{Hilbert schemes, polygraphs and the {M}acdonald positivity
  conjecture}, Journal of the American Mathematical Society \textbf{14} (2001),
  no.~4, 941--1006.

\bibitem[Hai01b]{Haiman01vanishingtheorems}
Mark Haiman, \emph{Vanishing theorems and character formulas for the {H}ilbert
  scheme of points in the plane}, Invent. Math \textbf{149} (2001), 371--407.

\bibitem[HHL{\etalchar{+}}05a]{haglund2005combinatoriala}
J.~Haglund, M.~Haiman, N.~Loehr, J.~B. Remmel, and A.~Ulyanov, \emph{A
  combinatorial formula for the character of the diagonal coinvariants}, Duke
  Math. J. \textbf{126} (2005), no.~2, 195--232. \MR{2115257}

\bibitem[HHL05b]{haglund2005combinatorial}
Jim Haglund, Mark Haiman, and Nick Loehr, \emph{A combinatorial formula for
  {M}acdonald polynomials}, Journal of the American Mathematical Society
  \textbf{18} (2005), no.~3, 735--761.

\bibitem[Hik14]{hikita2014affine}
T.~Hikita, \emph{Affine springer fibers of type $a$ and combinatorics of
  diagonal coinvariants}, Adv. Math. \textbf{263} (2014), 88--122.

\bibitem[Kiv20]{kivinen2020unramified}
Oscar Kivinen, \emph{Unramified affine {S}pringer fibers and isospectral
  hilbert schemes}, Selecta Mathematica \textbf{26} (2020), 42.

\bibitem[KK86]{kostant1986nil}
Bertram Kostant and Shrawan Kumar, \emph{The nil {H}ecke ring and cohomology of
  {$G/P$} for a {K}a\c{c}-{M}oody group {$G$}}, Proceedings of the National
  Academy of Sciences of the United States of America \textbf{83} (1986),
  1543--5.

\bibitem[Knu03]{knutson2003schubert}
Allen Knutson, \emph{A schubert calculus recurrence from the noncomplex
  w-action on g/b}, arXiv:math/0306304 (2003), 10.

\bibitem[LLM{\etalchar{+}}14]{lam2014schur}
Thomas Lam, Luc Lapointe, Jennifer Morse, Anne Schilling, Mark Shimozono, and
  Mike Zabrocki, \emph{$k$-{S}chur functions and affine {S}chubert calculus},
  vol.~33, Springer, 2014.

\bibitem[LS91]{lusztig1991fixed}
G.~Lusztig and J.~M. Smelt, \emph{Fixed point varieties on the space of
  lattices}, Bulletin of The London Mathematical Society \textbf{23} (1991),
  213--218.

\bibitem[Mac95]{macdonald1995symmetric}
Ian~G. Macdonald, \emph{Symmetric functions and {H}all polynomials}, second
  ed., Oxford Mathematical Monographs, The Clarendon Press, Oxford University
  Press, New York, 1995, With contributions by A. Zelevinsky, Oxford Science
  Publications. \MR{1354144}

\bibitem[Mel16]{mellit2016toric}
Anton Mellit, \emph{Toric braids and $(m, n)$-parking functions}, arXiv
  preprint arXiv:1604.07456 (2016), 37.

\bibitem[Mel17]{mellit2017poincarea}
\bysame, \emph{Poincare polynomials of character varieties, {M}acdonald
  polynomials and affine {S}pringer fibers}, arXiv preprint arXiv:1710.04513
  (2017), 64.

\bibitem[OY14]{oblomkov2014geometric}
Alexei Oblomkov and Zhiwei Yun, \emph{Geometric representations of graded and
  rational {C}herednik algebras}, Advances in Mathematics \textbf{292} (2014),
  83.

\bibitem[SW12]{shareshian2012chromatic}
John Shareshian and Michelle~L. Wachs, \emph{Chromatic quasisymmetric functions
  and {H}essenberg varieties}, Configuration Spaces (Pisa) (A.~Bjorner,
  F.~Cohen, C.~De~Concini, C.~Procesi, and M.~Salvetti, eds.), Scuola Normale
  Superiore, 2012, pp.~433--460.

\bibitem[Tym08]{tymoczko2008permutation}
Julianna~S. Tymoczko, \emph{Permutation representations on {S}chubert
  varieties}, American Journal of Mathematics \textbf{130} (2008), no.~5,
  1171--1194.

\end{thebibliography}

\end{document}